\documentclass{amsart}
\usepackage{etoolbox}
\makeatletter
\patchcmd{\subsubsection}{\itshape}{\bfseries}{}{}
\makeatother
\usepackage[shortlabels]{enumitem}
\usepackage[a4paper,margin=1in]{geometry} 
\usepackage{amsthm}
\usepackage{amssymb}
\usepackage{amsmath}
\usepackage{amsfonts}
\usepackage[all]{xy}
\usepackage{fancyhdr}
\usepackage{latexsym}
\usepackage{tikz}
\usepackage{mathrsfs}
\usepackage{epsfig}
\usepackage{graphicx}
\usepackage{caption}
\usepackage{subcaption}
\usepackage{float}
\usepackage{nomencl}
\usepackage{mathtools}
\usepackage{upgreek, bm, color, psfrag, euscript, amscd}
\usepackage{mathrsfs}
\usepackage{braket}
\usepackage{mathtools}
\usepackage{tikz-cd}
\usepackage{setspace}
\usepackage{xcolor}
\usepackage{textcomp}
\usepackage{comment}
\usepackage{blindtext}
\usepackage[
colorlinks=true,
linkcolor=black,      
citecolor=black,      
urlcolor=blue,        
filecolor=black,      
bookmarks=true,
CJKbookmarks=true
]{hyperref}
\usepackage{cleveref}
\usepackage{relsize}
\usepackage{spectralsequences}
\usepackage{url}
\usepackage{amsmath}

\newcommand{\N}{\mathbb{N}}
\newcommand{\Z}{\mathbb{Z}}
\newcommand{\Q}{\mathbb{Q}}
\newcommand{\R}{\mathbb{R}}

\renewcommand{\o}{\omega}
\newcommand{\Id}{\mathrm{Id}}
\newcommand{\Sign}{\mathrm{Sign}}

\newcommand{\ld}{\lambda}
\newcommand{\vp}{\varphi}
\newcommand{\g}{\gamma}

\newcommand{\w}{\wedge}
\newcommand{\pp}{\partial}

\newcommand{\eps}{\varepsilon}

\newtheorem{mainthm}{Main Theorem}
\newtheorem{thm}{Theorem}[section]
\newtheorem{lmm}[thm]{Lemma}
\newtheorem{prop}[thm]{Proposition}

\newtheorem{cor}[thm]{Corollary}
\newtheorem{note}[thm]{Note}

\newtheorem{remark}[thm]{Remark}

\begin{document}

\title{Conley-Zehnder Indices of Spatial Rotating Kepler Problem}
\author{Dongho Lee}
\maketitle

\begin{abstract}
	We study periodic orbits in the spatial rotating Kepler problem from a symplectic-topological perspective. Our first main result provides a complete classification of these orbits via a natural parametrization of the space of Kepler orbits, using angular momentum and the Laplace–Runge–Lenz vector. We then compute the Conley–Zehnder indices of non-degenerate orbits and the Robbin–Salamon indices of degenerate families, establishing their contributions to symplectic homology via the Morse–Bott spectral sequence. To address coordinate degeneracies in the spatial setting, we introduce a new coordinate system based on the Laplace–Runge–Lenz vector. These results offer a full symplectic-topological profile of the three-dimensional rotating Kepler problem and connect it to generators of symplectic homology.
\end{abstract}

\tableofcontents

\section{Introduction}

The study of periodic solutions in celestial mechanics has long served as a bridge between classical mechanics and symplectic geometry. Among various models, the rotating Kepler problem has provided a particularly rich source of examples for both dynamical systems and symplectic topologists.

Our first main result is a complete classification of periodic orbits of the rotating Kepler problem, based on a natural parametrization of the moduli space of the periodic orbits of Kepler problem, using the angular momentum and the Laplace-Runge-Lenz vector.
\begin{mainthm}
	Let $\mathcal{M}_E$ be a moduli space of the periodic orbits of Kepler problem with Kepler energy $E<0$.
	Let $L$ be the angular momentum and $A$ be the Laplace-Runge-Lenz vector.
	Then there exists a well-defined bijection
		$$
	\begin{aligned}
		\Phi: \mathcal{M}_E &\to S^2 \times S^2 \subset \R^3\times \R^3 \\
		\gamma &\mapsto \left( \sqrt{-2E} L - A, \, \sqrt{-2E} L + A \right).
	\end{aligned}
	$$
\end{mainthm}
While the topology of this moduli space, essentially the space of closed geodesics on $S^3$, is already well understood, our parametrization offers a new perspective rooted in the dynamics of the rotating Kepler problem.
This enables us to describe periodic orbits of the rotating Kepler problem in a convenient way.
For example, the third component of the angular momentum $L_3$ serves as a Morse function on $\mathcal{M}_E$.
The retrograde and direct orbits correspond to the maximum and minimum of $L_3$, and the Morse-Bott family of periodic orbits of the rotating Kepler problem appears as a regular level set of $L_3$, which is homeomorphic to $S^3$.
The detailed discussion can be found in \Cref{Subsection - Moduli}.

The second main result concerns the computation of the Conley-Zehnder indices of these periodic orbits, regarded as Reeb orbits in the cotangent bundle $T^*S^3$. The Conley-Zehnder indices for the planar rotating Kepler problem were previously computed in \cite{Albers_Fish_Frauenfelder_vanKoert_13}. Our work extends this computation to the spatial setting, including new periodic orbits which appear in the spatial problem.
\begin{mainthm}
		Let $c<-3/2$ be given and consider the periodic orbits of the rotating Kepler problem with Jacobi energy $c$.
	\begin{enumerate}
		\item The Conley-Zehnder indices of the $N$-th iterate of the retrograde orbit $\g_+$ and the direct orbit $\g_-$ are given by
			$$
			\mu_{CZ}(\g_\pm^N)=2+4\left\lfloor N\frac{(-2E)^{3/2}}{(-2E)^{3/2}\pm1}\right\rfloor.
			$$
		\item The Conley-Zehnder indices of the $N$-the iterate of the vertical collision orbits $\g_{c_\pm}$ are given by
			$$
			\mu_{CZ}(\g_{c_\pm}^N)=4N.
			$$
		\item The orbits with Kepler energy $E_{k,l}=-\frac{1}{2}\left(\frac{k}{l}\right)^{3/2}$ for some $k,l\in\N$ are periodic and form a Morse-Bott family $\Sigma_{k,l}$ which is homeomorphic to $S^3\times S^1$.
		The Robbin-Salamon index of $\Sigma_{k,l}$ is given by
			$$
			\mu_{RS}(\Sigma_{k,l})=4k-1/2.
			$$
	\end{enumerate}
\end{mainthm}
In \Cref{Theorem - Index of planar circular orbits}, we compute the indices of the retrograde and direct orbits, which also appear in the planar case.
There exists a degree shift due to the spatial direction, and the resulting indices are twice compared to the planar case.
In \Cref{Theorem - Index of Collision}, we compute the indices of vertical collision orbits, which are unique to the spatial problem. As shown in \Cref{Subsection - Relation with SH}, these indices align with the grading of the symplectic homology of $T^*S^3$, suggesting that the spatial rotating Kepler problem provides a geometric realization of certain generators in symplectic homology.

For the families of degenerate orbits, we compute their Robbin-Salamon indices and analyze their contributions to symplectic homology via the Morse-Bott spectral sequence, as detailed in \Cref{Theorem - Index of Degenerate Orbits}. A key step is to establish that these families form Morse-Bott type critical manifolds, which we prove in \Cref{Proposition - Morse-Bott property}. While Delaunay coordinates were effectively used in \cite{Albers_Fish_Frauenfelder_vanKoert_13} to demonstrate the Morse-Bott property in the planar setting, they become partially degenerate in the spatial case. To overcome this, we introduce a new coordinate system based on the Laplace-Runge-Lenz vector in \Cref{Subsection - Morse-Bott property}.

This work contributes to a growing body of results at the intersection of classical mechanics, symplectic topology, and Floer-theoretic invariants. It provides a complete symplectic-topological profile of the three-dimensional rotating Kepler problem, with implications for understanding periodic orbits in other central force systems and their symplectic interpretations.\\
~~\\
\textbf{Acknowledgement.} The author was partially supported by the Center for Quantum Structures in Modules and Spaces (QSMS), and by the National Research Foundation of Korea (NRF), Grant number (MSIT) RS-2023-NR076656 and NRF-2022R1F1A1074587.
The author would like to thank Professor Cheol-hyun Cho\footnotemark[1], professor Otto van Koert\footnotemark[1] and professor Joontae Kim\footnotemark[2] for their valuable support and advice.
Special thanks are due to Beomjun Sohn\footnotemark[1] for insightful discussions and Chankyu Joung\footnotemark[1] for preparing the illustration used in this work.
\footnotetext[1]{Department of Mathematics, Seoul National University, Korea}
\footnotetext[2]{Department of Mathematics, Sogang University, Korea}

\section{The Rotating Kepler Problem}

\subsection{Invariants of the Kepler Problem}\label{Subsection - Invariants}

The \textbf{Kepler problem} is a dynamical system that describes the motion of a massless body in Euclidean space under the gravitational force of another body.
For simplicity, we assume that the source of the gravitational force has unit mass and is fixed at the origin.
The Hamiltonian governing the Kepler problem is given by
$$
\begin{aligned}
	E:T^*(\R^3 \setminus\{0\})&\to \R\\
	(q,p)&\mapsto \frac{1}{2}|p|^2 - \frac{1}{|q|}.
\end{aligned}
$$
In the 17th century, Johannes Kepler established three laws of planetary motion based purely on observational data.
Deriving these laws is now a standard exercise in classical mechanics.
\begin{enumerate}
	\item Each orbit is a conic section with one focus at the origin. If $E < 0$, such an orbit is an ellipse, which can be possibly degenerate (the collision orbit, which will be explained in \Cref{Subsection - Moser regularization}).
	\item The areal velocity $d\text{Area}/dt = r^2\dot{\theta}$ is constant, where $(r, \theta)$ are polar coordinates on the plane containing the orbit.
	\item Let $\tau$ denote the orbital period. Then $\tau = 2\pi / (-2E)^{3/2}$.
\end{enumerate}
We now introduce two fundamental invariants associated with the Kepler problem.
The first is the \textbf{angular momentum}, defined by
	$$L = (L_1,L_2,L_3) = q \times p,$$
where $\times$ denotes the cross product in $\R^3$.
By "invariant", we mean that $L_i$'s are conserved quantities along the $E$-flow, i.e., $\{E, L_i\} = 0$ for each $i = 1,2,3$.
This invariance follows from the obvious $SO(3)$-symmetry of the system. The Hamiltonian flow generated by each $L_i$ corresponds to a rotation of period $2\pi$ about the $q_i$- and $p_i$-axes.

Another important invariant is the \textbf{Laplace–Runge–Lenz (LRL) vector}, defined by
$$
A = p\times L - \frac{q}{|q|}.
$$
Here are some well-known properties of LRL vector, which will be used in later sections.

\begin{prop}\label{Proposition - LRL properties}
	Let $A$ be the Laplace–Runge–Lenz vector, and assume that the Kepler energy $E$ is negative, so that every orbit is bounded. Then,
	\begin{enumerate}
		\item $A$ is an invariant of the system: $\{E, A_i\} = 0$ for each $i = 1,2,3$.
		\item $A$ is orthogonal to the angular momentum $L$.
		\item $A$ is aligned with the major axis of the elliptic orbit.
		\item Let $\varepsilon$ denote the eccentricity of the ellipse. Then
			$$
			\eps^2 = |A|^2 = 2E|L|^2 + 1.
			$$
		\item $\{A_i, L_j\} = \sum_k \varepsilon_{ijk} A_k$, where $\varepsilon_{ijk}$ is the Levi-Civita symbol \footnotemark[3].
	\end{enumerate}
\end{prop}
\footnotetext[3]{$\eps_{ijk}$ defined by the sign of a permutation $(i,j,k)$. It can be simply computed by $\det(e_i,e_j,e_k)$, where $e_i$ is $i$-th standard basis of $\R^3$.}

An important consequence is that the triple $(E, L, A)$ uniquely determines the Kepler orbit, as illustrated in \Cref{Figure : Kepler Orbit}.

\begin{prop}\label{Proposition - ELA Kepler orbit}
	Given a Kepler orbit (and hence given $L$ and $A$), let $(r,\theta)$ denote polar coordinates in the plane defined by $L \cdot  q = 0$.
	Then the orbital trace is given by
		$$
		r = \frac{|L|^2}{1+|A|\cos(\theta-g)},
		$$
	where $g$, the \emph{argument of perigee}, is determined by $L$ and $A$.
\end{prop}

\begin{figure}[ht]
	\centering
	\includegraphics[width=7cm]{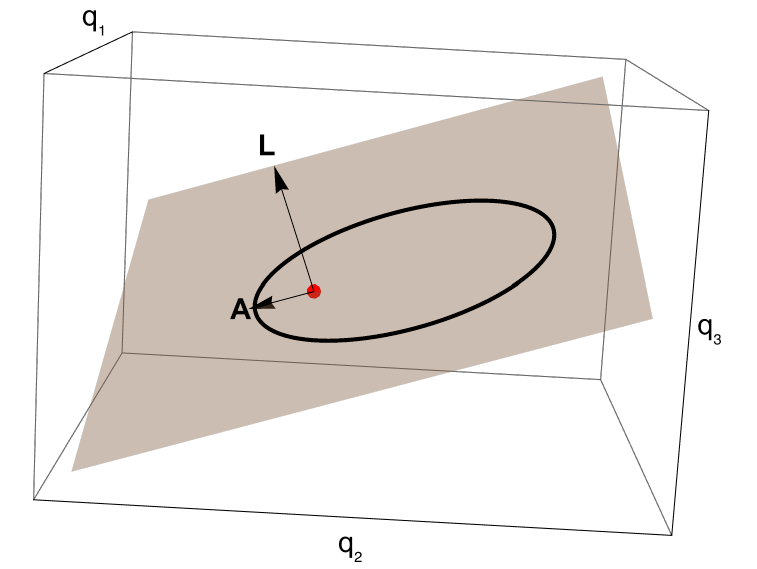}
	\caption{Kepler orbit determined by $L$ and $A$.}
	\label{Figure : Kepler Orbit}
\end{figure}

The \textbf{rotating Kepler problem} is defined by the following Hamiltonian
	$$
		H = E+L_3 = \frac{1}{2}|p|^2 - \frac{1}{|q|}+(q_1p_2 - q_2p_1).
	$$
To avoid confusion, we refer to $H$ as the \textbf{Jacobi energy}, and to $E$ as the \textbf{Kepler energy}.

	\subsubsection*{Motivation of the Rotating Kepler Problem}
	A natural motivation for considering the rotating Kepler problem is the circular restricted three-body problem.
	This models the motion of a massless body under the gravitational influence of two massive bodies, typically referred to as the Earth and the Moon.
	Let their masses have the ratio $1-\mu : \mu$, where $0 < \mu < 1$, and let $e(t)$ and $m(t)$ denote their trajectories. Then the system is governed by the time-dependent Hamiltonian
	$$
	H_t(q,p) = \frac{1}{2}|p|^2 - \frac{\mu}{|q-m(t)|}-\frac{1-\mu}{|q-e(t)|}.\\
	$$
	If we assume the Earth and Moon move in circular orbits, then
	$$
	e(t)=-\mu(\cos t,-\sin t,0),\quad m(t) = (1-\mu)(\cos t, -\sin t, 0),
	$$
	In this case, the system becomes the circular restricted three-body problem.
	By changing to a rotating reference frame in which the Earth and Moon are fixed at $(-\mu, 0)$ and $(1 - \mu, 0)$, respectively, the Hamiltonian becomes autonomous and takes the form
	$$
	H(q,p) = \frac{1}{2}|p|^2 -\frac{1-\mu}{|q+\mu|}-\frac{\mu}{|q-(1-\mu)|}+ (q_1 p_2 - q_2 p_1).
	$$
	Taking the limit $\mu \to 0$, which corresponds to the Moon having negligible mass, yields the Hamiltonian for the rotating Kepler problem.


\subsection{Moser Regularization}\label{Subsection - Moser regularization}

The singularity at the origin in the Kepler problem prevents the level set from being compact or complete.
To resolve this issue, we introduce the regularization method proposed by Moser in \cite{Moser_70}.
We begin with formulae for the stereographic projection.

\begin{lmm}\label{Lemma - Stereographic Projection}
Let $S^3_r$ denote the $3$-sphere of radius $r$, and parametrize $T^*S_r^3$ by
	$$
	T^*S_r^3 = \left\{(x;y) = (x_0,\vec{x};y_0,\vec{y})\in \R^4\times \R^4\,:\,|x|^2=1,\,x\cdot y=0\right\}.
	$$
The stereographic projection from the north pole $(r,0,0,0)$ is given by
	$$
	\begin{aligned}
		\Phi_r:T^*S_r^3&\to T^*\R^3\\
		(x,y)&\mapsto\left(\frac{r\vec{x}}{r-x_0},\frac{r-x_0}{r}\vec{y}+\frac{y_0}{r}\vec{x}\right),
	\end{aligned}
	$$
with the inverse
	$$
	\begin{aligned}
		\Psi_r:T^*\R^3&\to T^*S_r^3\\
		(p,q)&\mapsto
		\left(\frac{r(|p|^2-r^2)}{|p|^2+r^2},\frac{2r^2p}{|p|^2+r^2},\frac{p\cdot q}{r},\frac{|p|^2+r^2}{2r^2}q - \frac{p\cdot q}{r^2}p\right).
	\end{aligned}
	$$
In particular, the following relations hold.
	$$
	\begin{aligned}
		r-x_0&=\frac{2r^3}{|p|^2+r^2}, &|p|^2&=\frac{2r^3}{r-x_0}-r^2,\\
		|y|^2 & = \frac{(|p|^2+r^2)^2}{4r^4}|q|^2,&|q|&=\frac{2r^2}{|p|^2+r^2}|y|=\frac{r-x_0}{r}|y|.
	\end{aligned}
	$$	
\end{lmm}

The Moser regularization of the Kepler problem works as follows.
Fix an energy level $E = E_0 < 0$, and define the modified Hamiltonian on $T^*(\R^3 \setminus \{0\})$
$$
\tilde{K}_{E_0}(q,p) = \frac{1}{2}\left(|q|\left(E(q,p)-E_0\right)+1\right)^2
= \frac{1}{2}\left(\frac{1}{2}(|p|^2-2E_0)|q|\right)^2.
$$
This Hamiltonian extends smoothly to the origin, allowing it to be defined on all of $T^*\R^3$.
Moreover, the level set $E^{-1}(E_0)$ is contained in $\tilde{K}_{E_0}^{-1}(1/2)$, so the Hamiltonian flows of $E$ and $\tilde{K}_{E_0}$ are equivalent on this set, up to reparametrization.

Define the switch map $\sigma: T^*\R^3 \to T^*\R^3$ by $\sigma(q,p) = (p, -q)$, which is a symplectomorphism.
By composing $\tilde{K}_{E_0}$ with $\sigma$ and the inverse stereographic projection $\Psi_r$, we obtain a Hamiltonian on $T^*S^3_r$,
$$
K_r(x,y) = \tilde{K}_{E_0}(\Psi_r(p,-q))=\frac{1}{2}|y|^2 \left(r^2 + (r-x_0)\left(-\frac{E_0}{r^2}-\frac{1}{2}\right)\right)^2,
$$
where $(x,y)$ is coordinates of $T^*S^3_r$, used in \Cref{Lemma - Stereographic Projection}.
By setting $r = \sqrt{-2E_0}$, this simplifies to
$$
K_r(x,y) = \frac{r^4}{2}|y|^2.
$$
This defines the geodesic flow on $T^*S^3_r$.
In other words, the level set $E = E_0$ of the Kepler problem is embedded as a subsystem of the geodesic flow on $T^*S^3_r$.

For convenience, we often pull back $K_r$ to $T^*S^3 = T^*S^3_1$ using the scaling map
$$
\begin{aligned}
	sc_{1,r}:T^*S^3&\to T^* S_r^3\\
	(x,y)&\mapsto (r x, y/r),
\end{aligned}
$$
so that the Hamiltonian defined on $T^*S^3_1$ becomes
$$
K_r(x,y) = \frac{r^2}{2}|y|^2.
$$
The orbits added by Moser regularization correspond to those passing through the origin, and are called \textbf{collision orbits}.
For example, consider the initial condition
$$
(q(0);p(0)) = \left(0,0,-\frac{1}{E_0}\,;\,0,0,0\right)=\left(0,0,\frac{2}{r^2}\,;\,0,0,0\right).
$$
This represents an object falling directly into the origin and colliding in finite time.
In $T^*S^3$, the point $(q(0);p(0))$ corresponds to
$$
(x(0);y(0)) = \left(-1,0,0,0\,;\,0,0,0,-\frac{1}{r}\right).
$$
The orbit of the Hamiltonian vector field $X_{K_r}$ with this initial condition is given by
$$
(x(t);y(t)) = \left(-\cos (r t),0,0,-\sin (rt)\,;\,\frac{\sin (rt)}{r},0,0,-\frac{\cos (rt)}{r}\right).
$$
Applying stereographic projection and the inverse switch map yields the corresponding orbit in $T^*\R^3$:
$$
(q(t);p(t)) = \left(0,0,\frac{1}{r^2}(1+\cos(rt))\,;\,0,0,-\frac{r\sin(rt)}{1+\cos (rt)}\right).
$$
Any initial condition with $p(0) = 0$ leads to a similar trajectory.
The resulting collision orbit oscillates between the origin and a maximal height in $T^*\R^3$.

\begin{figure}[ht]
	\centering
	\includegraphics[width=7cm]{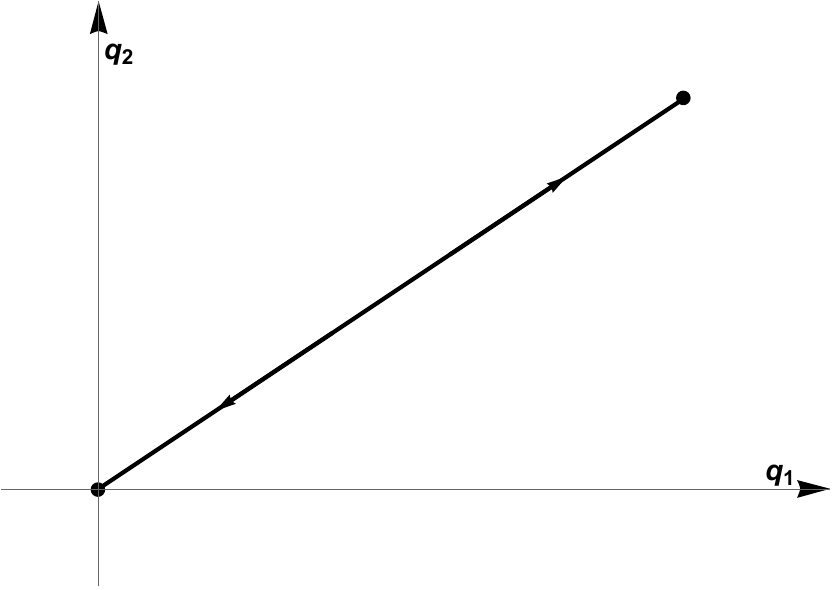}
	\caption{Collision orbit in the $q_1q_2$-plane.}
	\label{Figure : Collision Orbit}
\end{figure}

Note that this parametrization of the collision orbit corresponds to a reparametrization of the Kepler orbit, so the orbital speed differs.
In particular, the period of $\gamma$ as a Kepler orbit is not equal to $2\pi r$.
The energy hypersurface $K_r^{-1}(1/2)$, which now contains the collision orbits, is given by $S_rT^*S^3$.

\begin{note}\rm
	For collision orbits, $q$ and $p$ are always parallel.
	It follows that $L=0$ and $A=-q/|q|$.
	Even though $A$ is not defined at the origin, we can see that $A$ is constant along the collision orbit, so we can regard $A$ as an invariant after regularization.
	Moreover, regarding the formula in \Cref{Proposition - LRL properties} (4), we conclude that the collision orbit corresponds to the case $\eps=1$, even though the eccentricity of a straight line is infinity.
\end{note}

To regularize the rotating Kepler problem, we first rewrite the Hamiltonian
$$
H = \frac{1}{2}|\tilde{p}|^2 + U(q)
$$
where $\tilde{p} = (p_1 - q_2, p_2 + q_1, p_3)$, and the effective potential $U$ is defined by
	$$ U(q) = -\frac{1}{|q|} - \frac{1}{2}(q_1^2+q_2^2).$$
Let $\pi: T^*(\R^3 \setminus \{0\}) \to \R^3 \setminus \{0\}$ denote the canonical projection.
For a fixed energy level $c$, the \textbf{Hill's region} is defined as
$$
\mathcal{H}_c = \left\{q\,:\,U(q)\leq c \right\}.
$$
Since $|\tilde{p}|^2 \geq 0$, any point on the energy surface $H(q,p) \leq c$ must project to $\mathcal{H}_c$.
It can be shown that
\begin{itemize}
	\item If $c<-3/2$, then $\mathcal{H}_c$ consists of one bounded component (diffeomorphic to a 3-ball) and one unbounded component.
	\item If $c>-3/2$, then $\mathcal{H}_c = \R^3 \setminus \{0\}$.
\end{itemize}
Thus, the energy $c = -3/2$ is called the \textbf{critical energy}.
We are particularly interested in orbits with $c<-3/2$, lying within the bounded component of $\mathcal{H}_c$.

\begin{figure}[ht]
	\centering
	\includegraphics[width=6cm]{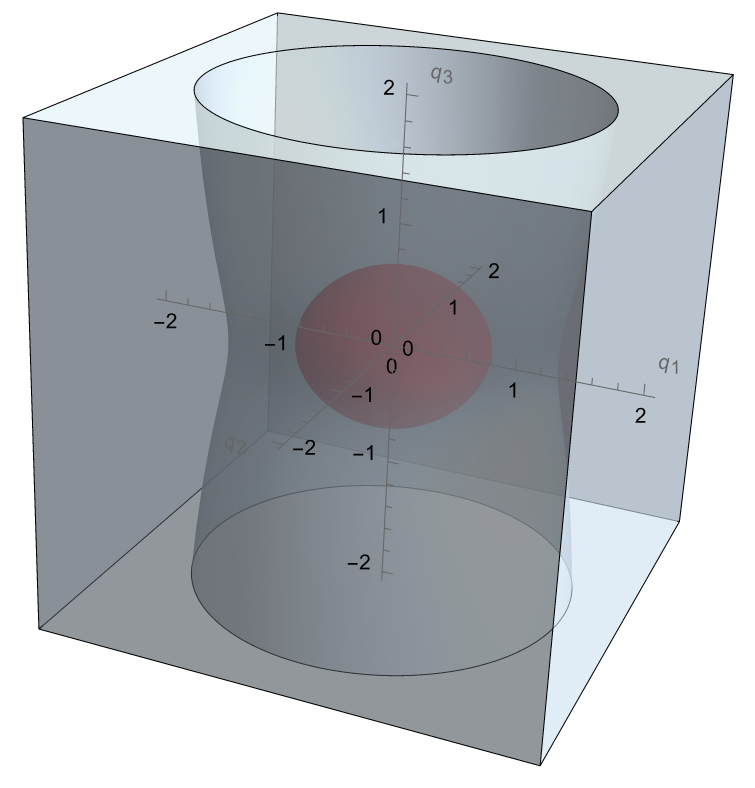}
	\caption{Hill's region. The inner red ball represents the bounded component.}
	\label{Figure : Hills region}
\end{figure}

Fix a Jacobi energy level $c < -3/2$.
We now apply Moser regularization to the spatial rotating Kepler problem, restricted to the bounded component of the Hill's region.
As before, the regularized Hamiltonian on $T^*S^3$ is
$$
K_{c}(x,y) = \frac{1}{2}|y|^2\left(r+\frac{1}{r^2}(1-x_0)(x_1 y_2 - x_2 y_1)\right)^2,
$$
where $r = \sqrt{-2c}$.
It is shown in \cite{Ceiliebak_Frauenfelder_vanKoert_14} that $K_c$ defines a Finsler geodesic flow on $T^*S^3$, and that the Hamiltonian flow of the rotating Kepler problem is embedded within the level set $K_c^{-1}(1/2)$.
As before, the energy hypersurface is diffeomorphic to the unit cotangent bundle $ST^*S^3$.


\section{Periodic Orbits of the Rotating Kepler Problem}\label{Section - Periodic Orbits}


\subsection{Classification}\label{Subsection - Classification of periodic orbits}

As discussed in \Cref{Subsection - Invariants}, we have $\{E, L_3\} = 0$. It follows that the Hamiltonian flow of $H$ satisfies
$$
Fl^{X_H}_t = Fl^{X_E}_t \circ Fl^{X_{L_3}}_t,
$$
where $X_H$ is a Hamiltonian vector field of $H$, and $Fl^{X}$ denotes the flow of a vector field $X$.
Since the flows of both $E$ and $L_3$ are known, we can analyze the orbits of $H$ in a straightforward manner.

\subsubsection*{Planar Circular Orbits}

There are two types of Kepler orbits that remain periodic in the rotating Kepler problem for a generic Jacobi energy level $c$.
The first type is the \textbf{planar circular orbit}.
Since the flow of $L_3$ corresponds to a rotation in the $q_3$- and $p_3$-directions, a circular orbit confined to the $q_1q_2$-plane becomes periodic after composing it with the $L_3$-flow.

To construct such a planar circular orbit at the energy level $H = c$, we require
$$
\varepsilon^2 = 2E|L|^2 + 1 = 2EL_3^2 + 1 = 2E(c - E)^2 + 1 = 0.
$$
As described in \Cref{Figure : PC orbit graphs}, a straightforward computation shows that
\begin{itemize}
	\item If $c < -3/2$, there are three possible values of $E$ satisfying the above equation.
	\item If $c > -3/2$, only one such value exists.
\end{itemize}
Assume $c < -3/2$, so that three distinct circular orbits exist.
\begin{enumerate}
	\item The \textbf{retrograde orbit}, denoted by $\gamma_+$, corresponds to the smallest Kepler energy. It satisfies $L_3 > 0$, so it has a smaller radius and rotates counterclockwise.
	From \Cref{Figure : PC orbit graphs}, we see that the retrograde orbit is the only one that persists beyond the critical energy level $c = -3/2$.
	\item The \textbf{direct orbit}, denoted by $\gamma_-$, corresponds to the second smallest Kepler energy. Here $L_3 < 0$, leading to a larger radius and clockwise rotation.
	\item The \textbf{outer direct orbit}, which has the largest Kepler energy, lies outside the bounded component of the Hill's region.
	Since our focus is on the bounded component under Moser regularization, we do not consider this orbit further.
\end{enumerate}

\begin{figure}[ht]
	\centering
	\includegraphics[width=7cm]{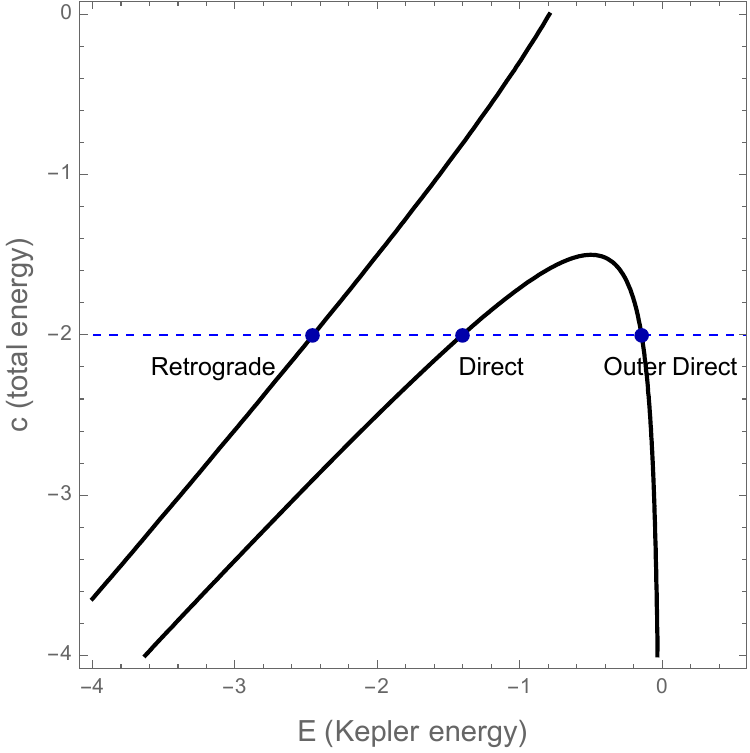}
	\caption{Graph of $2E(c - E)^2 + 1 = 0$.}
	\label{Figure : PC orbit graphs}
\end{figure}

\begin{lmm}\label{Lemma - Periods of Retrograde and Direct}
	Let $E_\pm$ and $\tau_\pm$ denote the Kepler energy and the period of the retrograde orbit $\gamma_+$ and the direct orbit $\gamma_-$, respectively, for a given Jacobi energy $c$. Then,
	$$
	c = E_\pm \pm \frac{1}{\sqrt{-2E_\pm}},\quad \tau_\pm = \frac{2\pi}{(-2E_\pm)^{3/2} \pm 1}.
	$$
	Moreover, an explicit parametrization of $\gamma_\pm$ under the cylindrical coordinate system on $T^*\R^3\setminus\{0\}$ is given by
	$$
	\gamma_\pm(t) = \begin{pmatrix}
		r(t)\\
		\theta(t)\\
		z(t)\\
		p_r(t)\\
		p_\theta(t)\\
		p_z(t)
	\end{pmatrix} = \begin{pmatrix}
		\omega_0^2\\
		\left(\frac{1}{\omega_0^3} + 1\right)t\\
		0\\
		0\\
		\omega_0\\
		0
	\end{pmatrix},\quad \omega_0 = \pm \frac{1}{\sqrt{-2E_\pm}}.
	$$
\end{lmm}

\begin{proof}
	For circular orbits, the condition $\varepsilon^2 = 2EL_3^2 + 1 = 0$ implies $L_3 = \pm 1/\sqrt{-2E}$, yielding the first identity.
	
	To compute the period, we use cylindrical coordinates
	$$(q_1, q_2, q_3) = (r\cos\theta, r\sin\theta, z),$$
	in which the Hamiltonian becomes:
	$$
	H(r, \theta, z; p_r, p_\theta, p_z) = \frac{1}{2} \left(p_r^2 + \frac{p_\theta^2}{r^2} + p_z^2\right) - \frac{1}{\sqrt{r^2 + z^2}} + p_\theta.
	$$
	The associated Hamiltonian vector field is
	$$
	X_H = p_r\partial_r + \left(\frac{p_\theta}{r^2} + 1\right)\partial_\theta + p_z\partial_z + \left(\frac{p_\theta^2}{r^3} - \frac{r}{(r^2 + z^2)^{3/2}}\right)\partial_{p_r} - \frac{z}{(r^2 + z^2)^{3/2}}\partial_{p_z}.
	$$
	Assuming a planar ($z = p_z = 0$) and circular orbit ($r = r_0$, $p_r = 0$, $p_\theta^2 = r_0$), we write $\omega_0 = \pm \sqrt{r_0}$ and note that $r_0 = -1/2E$. Then,
	$$
	X_H = \left(\frac{1}{\omega_0^3} + 1\right)\partial_\theta,
	$$
	which yields the stated period and parametrization.
\end{proof}
\subsubsection*{Vertical Collision Orbits}

The second type of periodic orbit consists of the vertical collision orbits, which arise in the Moser regularization. Although finding general solutions of the Hamiltonian equations associated with the regularized Hamiltonian $K_c$ is difficult, a substantial simplification occurs for collision orbits starting from initial conditions of the form
$$
(p(0);q(0)) = (0,0,P;\,0,0,Q).
$$
Under these conditions, the equations of motion reduce to
$$
\begin{pmatrix}
	\dot{x}_0\\
	\dot{x}_3\\
	\dot{y}_0\\
	\dot{y}_3
\end{pmatrix}
=
\begin{pmatrix}
	r^2 y_0\\
	r^2 y_3\\
	- x_0\\
	- x_3
\end{pmatrix},
$$
with $x_1 = x_2 = y_1 = y_2 = 0$ and $r=\sqrt{-2E}$.

These orbits are unaffected by the angular momentum $L_3$ and correspond to vertical collision orbits in the original Kepler problem, with initial conditions $q_3(0) = \mp 1/E$ and $p_3(0) = 0$. We denote these orbits by $\gamma_{c_\pm}$, referring to them as the \textbf{positive} and \textbf{negative vertical collision orbits}. These, together with the planar circular orbits, are shown in \Cref{Figure : Nondegenerate orbits}.

Note that the angular momentum of each collision orbit is zero, which implies $c = E$ for these orbits. As a result, they can be represented as points along the diagonal $c=E$ in the $(c,E)$-plane, as seen in \Cref{Figure : Orbit Diagram}. For generic energy levels $c$, the vertical collision orbits are isolated.

\begin{figure}[ht]
	\centering
	\begin{subfigure}[b]{0.45\textwidth}
		\centering
		\includegraphics[width=\textwidth]{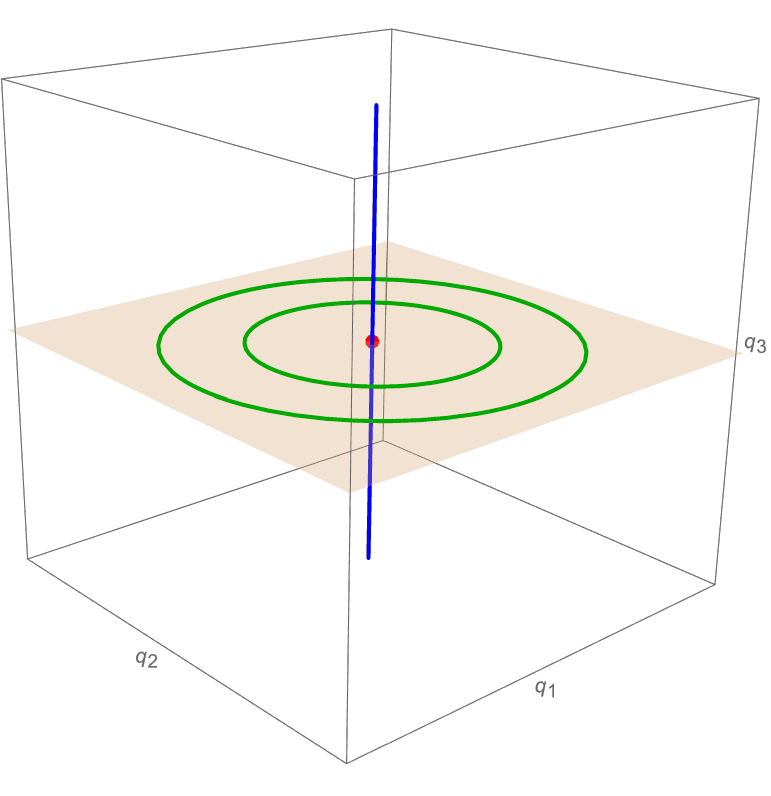}
		\caption{Planar circular orbits (green) and vertical collision orbits (blue).}
		\label{Figure : Nondegenerate orbits}
	\end{subfigure}
	\hfill
	\begin{subfigure}[b]{0.45\textwidth}
		\centering
		\includegraphics[width=\textwidth]{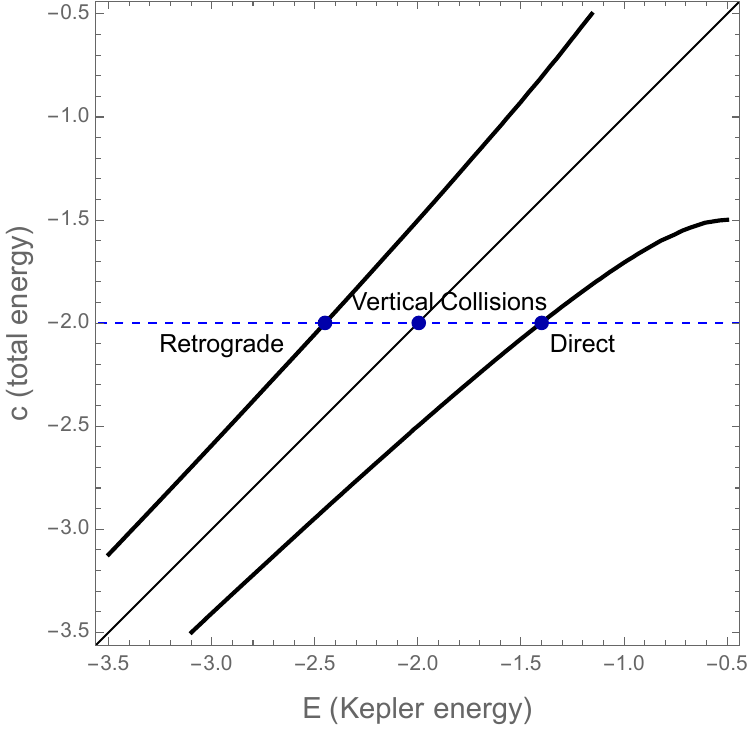}
		\caption{Non-degenerate orbits marked in the $(c,E)$-plane.}
		\label{Figure : Orbit Diagram}
	\end{subfigure}
	\caption{Non-degenerate periodic orbits.}
	\label{Figure : CombinedOrbits}
\end{figure}

\subsubsection*{Morse-Bott Family}

The remaining orbits are generically not periodic, but they may become periodic when a resonance condition is satisfied. According to Kepler’s third law, the period $\tau$ of the $E$-flow is given by
$$
\tau = \frac{2\pi}{(-2E)^{3/2}}.
$$
On the other hand, the period of the $L_3$-flow is $2\pi$. Since the total flow decomposes as $Fl^{X_H}_t = Fl^{X_E}_t \circ Fl^{X_{L_3}}_t$, a resonant orbit arises when there exist $k, l \in \N$ such that $k\tau = 2\pi l$, which implies
$$
E = E_{k,l} = -\frac{1}{2} \left(\frac{k}{l}\right)^{2/3}.
$$

That is, if a Kepler orbit has Kepler energy $E_{k,l}$, then its $k$-fold cover composed with the $l$-fold iteration of the $L_3$-flow becomes periodic. Once $E_{k,l}$ is fixed, the angular momentum $L_3$ is determined by $c - E_{k,l}$. The family of orbits with the same $E$ and $L_3$ forms a family.
Such family is a Morse-Bott family diffeomorphic to $S^3\times S^1$, which will be verified in \Cref{Subsection - Moduli} and \Cref{Subsection - Morse-Bott property}.
We denote such a family by $\Sigma_{k,l}$.
Examples of these orbits are illustrated in \Cref{Figure : Degenerate Orbits}.\footnotemark[4]

\begin{figure}[ht]
	\centering
	\begin{subfigure}[b]{0.45\textwidth}
		\centering
		\includegraphics[width=\textwidth]{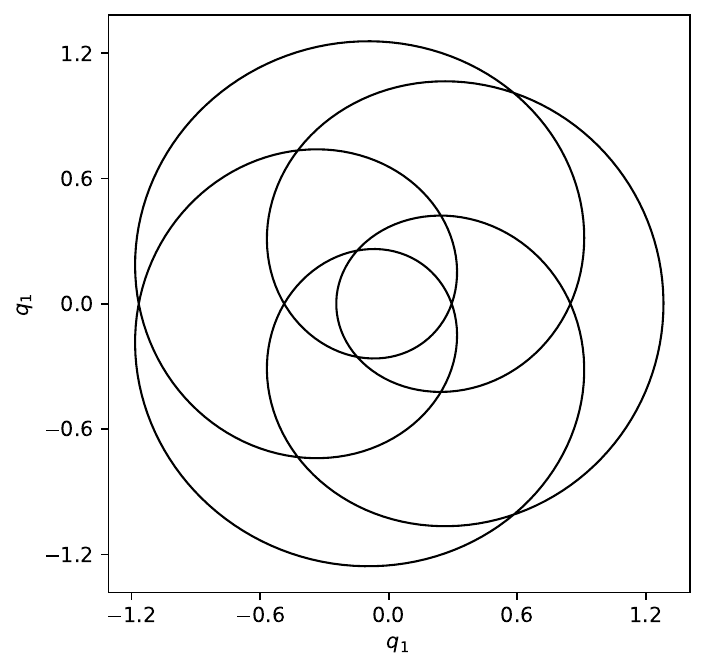}
	\end{subfigure}
	\hfill
	\begin{subfigure}[b]{0.45\textwidth}
		\centering
		\includegraphics[width=\textwidth]{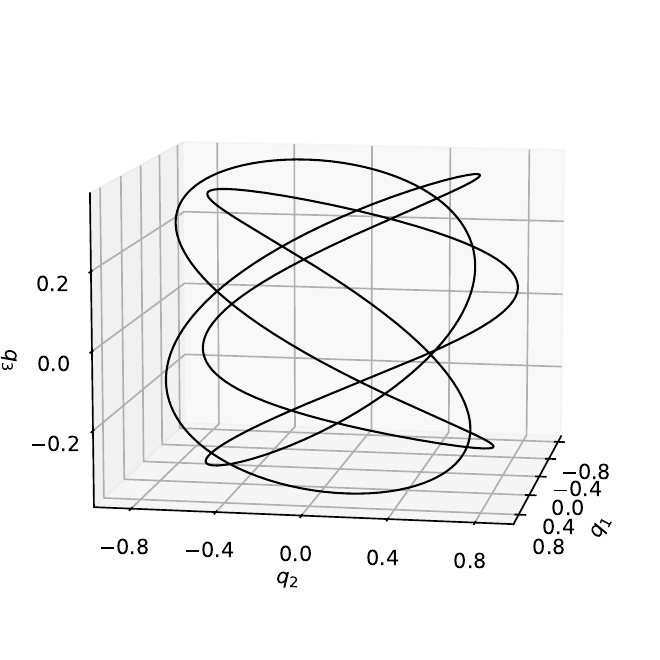}
	\end{subfigure}
	\caption{$\Sigma_{k,l}$-type orbits in the plane and in space.}
	\label{Figure : Degenerate Orbits}
\end{figure}

\footnotetext[4]{Figure courtesy of Chankyu Joung.}.

To summarize, fix a generic energy level $c$, which means that $c \neq E_{k,l}$ and $c\neq E_{k,l}\pm 1/\sqrt{-2E_{k,l}}$ for any $k,l \in \N$.
Under this assumption, the orbits $\gamma_\pm$ and $\gamma_{c_\pm}$ are not contained in any $\Sigma_{k,l}$ family and are thus isolated.

Then, the energy hypersurface $H^{-1}(c)$ contains the following periodic orbits and their multiple covers:
\begin{enumerate}
	\item Two planar circular orbits $\gamma_\pm$ with Kepler energies $E_\pm$ satisfying
	$$
	E_\pm \pm \frac{1}{\sqrt{-2E_\pm}} = c.
	$$
	\item Two vertical collision orbits $\gamma_{c_\pm}$ with Kepler energy $E = c$.
	\item Morse-Bott family $\Sigma_{k,l}$ for each pair $(k,l)$ such that $E_+ < E_{k,l} < E_-$.
\end{enumerate}
This structure is illustrated in \Cref{Figure : Energy Hypersurface}.

\begin{figure}[ht]
	\centering
	\includegraphics[width=7cm]{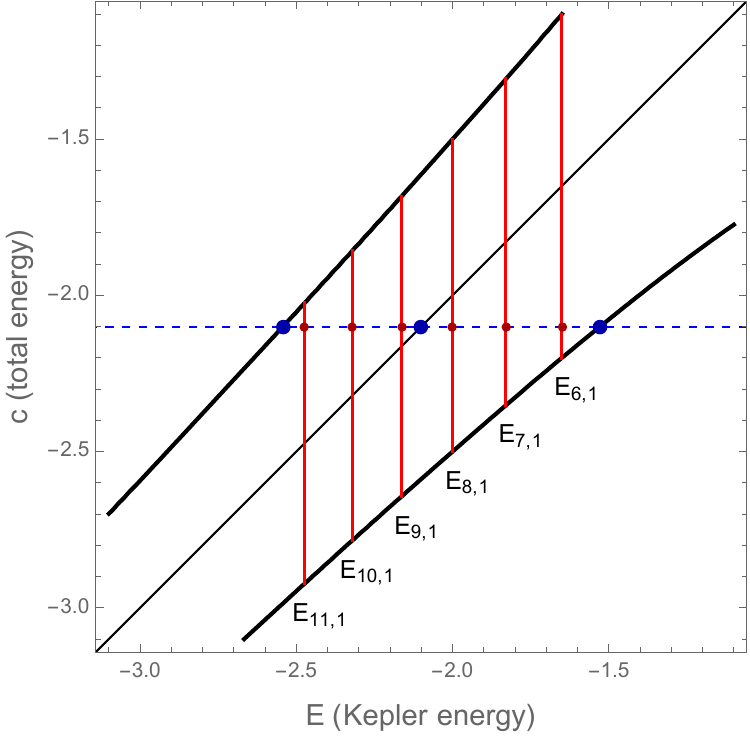}
	\caption{The energy hypersurface $H^{-1}(-2.1)$ containing four isolated orbits indicated by blue dots, and a countable collection of $\Sigma_{k,l}$-families, partially indicated by red dots.}
	\label{Figure : Energy Hypersurface}
\end{figure}


\subsection{Moduli Space of Kepler Orbits}\label{Subsection - Moduli}

The regularized Kepler problem can be identified with the geodesic flow on $T^*S^3$.
Hence, the total orbit space corresponds to the unit cotangent bundle $ST^*S^3$, which is diffeomorphic to $S^3 \times S^2$.
This implies that the moduli space of Kepler orbits can be identified with the space of unit geodesics on $S^3$.
Similarly, the moduli space of the planar Kepler orbits can be identified with the space of unit geodesics on $S^2$.
A classical result characterizes the moduli space of unit geodesics on the spheres.

\begin{thm}\label{Theorem - Moduli space of unit geodesics of spheres}
	The moduli space of unit geodesics on $S^2$ is diffeomorphic to $S^2$, and on $S^3$ it is diffeomorphic to $S^2 \times S^2$.
\end{thm}

\begin{proof}
	See \cite{Besse_78}, Propositions 2.9 and 2.10.
\end{proof}

We now rederive this result in a manner suited to the Kepler problem.
As shown in \Cref{Proposition - ELA Kepler orbit}, each Kepler orbit is completely characterized by its energy $E$, angular momentum $L$, and Laplace-Runge-Lenz vector $A$.
\begin{remark}\rm
Geometrically, an ellipse in $\mathbb{R}^3$ with one focus at the origin is determined by the orbital plane and orientation, given by $L$, the major axis, given by $A$, and the eccentricity $\varepsilon$, given by $\varepsilon^2 = |A|^2 = 2E|L|^2 + 1$.
In particular, a collision orbit corresponds to the case $\varepsilon = 1$, or equivalently, $L = 0$.
\end{remark}

Let $E < 0$ be the given Kepler energy, and let $\mathcal{M}_E$ denote the level set $K_E^{-1}(1/2)$ of the regularized Hamiltonian $K_E : T^*S^3 \to \mathbb{R}$ modulo the $S^1$-action induced by the Hamiltonian flow.
In other words, a point in $\mathcal{M}_E$ represents a simple Kepler orbit, possibly a collision.
This is the moduli space we aim to describe.

\begin{thm}\label{Theorem - Moduli parametrization}
	The map
	$$
	\begin{aligned}
		\Phi: \mathcal{M}_E &\to S^2 \times S^2 \subset \R^3\times \R^3 \\
		\gamma &\mapsto \left( \sqrt{-2E} L - A, \, \sqrt{-2E} L + A \right)
	\end{aligned}
	$$
	is a well-defined bijection that parametrizes the space of closed Kepler orbits with Kepler energy $E$.
\end{thm}

\begin{proof}
	Let $(x, y) = \left( \sqrt{-2E}L - A, \sqrt{-2E}L + A \right)$.
	Then we compute
	$$
	\begin{aligned}
		|x|^2 &= -2E|L|^2 + |A|^2 - 2\sqrt{-2E} (L\cdot A)  = -2E|L|^2 + |A|^2 = 1, \\
		|y|^2 &= -2E|L|^2 + |A|^2 + 2\sqrt{-2E} (L \cdot A) = -2E|L|^2 + |A|^2 = 1.
	\end{aligned}
	$$
	Hence, $(x, y) \in S^2 \times S^2$.
	Given $(x, y)$, we recover
	$$
	L = \frac{x + y}{2\sqrt{-2E}}, \qquad A = -\frac{x - y}{2}.
	$$
	From $L$ and $A$, the orbit $\gamma$ can be uniquely reconstructed via \Cref{Proposition - ELA Kepler orbit}.
\end{proof}

This explicit parametrization $\mathcal{M}_E \cong S^2 \times S^2$ allows us to describe special classes of orbits as follows.

\begin{enumerate}
	\item Circular orbits correspond to the diagonal $\{ x = y \}$, i.e., $A = 0$, forming an $S^2$-family.
	
	\item Collision orbits correspond to the anti-diagonal $\{ x = -y \}$, i.e., $L = 0$, also forming an $S^2$-family.
	
	\item Planar orbits satisfy $L_1 = L_2 = A_3 = 0$, which gives
	$$
	\mathcal{N}_E = \left\{ (x, y) \in S^2 \times S^2 : x_1 = -y_1,\; x_2 = -y_2,\; x_3 = y_3 \right\} \cong S^2.
	$$
	
	\item \textbf{Vertical orbits}, orthogonal to the $q_1q_2$-plane, satisfy $A_1 = A_2 = L_3 = 0$, i.e.,
	$$
	\mathcal{V}_E = \left\{ (x, y) \in S^2 \times S^2 : x_1 = y_1,\; x_2 = y_2,\; x_3 = -y_3 \right\} \cong S^2.
	$$
	
	\item Retrograde and direct circular orbits correspond to
	$$
	\gamma_+ = ((0,0,1), (0,0,1)), \quad \gamma_- = ((0,0,-1), (0,0,-1)).
	$$
	
	\item Vertical collision orbits correspond to
	$$
	\gamma_{c_+} = ((0,0,1), (0,0,-1)), \quad \gamma_{c_-} = ((0,0,-1), (0,0,1)).
	$$
\end{enumerate}

We now turn to the topology of $\mathcal{M}_E$.
Let us view the third component of angular momentum, $L_3$, as a Morse function on $\mathcal{M}_E$
$$
\begin{aligned}
	L_3: \mathcal{M}_E &\to \mathbb{R} \\
	(x, y) &\mapsto \frac{x_3 + y_3}{2\sqrt{-2E}}.
\end{aligned}
$$

Then $L_3$ is a Morse function with exactly four critical points,
\begin{enumerate}
	\item $\gamma_-$: the direct orbit - global minimum,
	\item $\gamma_{c_\pm}$: the vertical collision orbits - index 2 saddle points,
	\item $\gamma_+$: the retrograde orbit - global maximum.
\end{enumerate}

From the perspective of Morse theory, each regular level set of $L_3$ is diffeomorphic to $S^3$.
Thus, the family $\Sigma_{k,l}$ of periodic orbits corresponding to energy $E_{k,l}$ and angular momentum $L_3 = c - E_{k,l}$, discussed in \Cref{Subsection - Classification of periodic orbits}, is topologically $S^3$.

\begin{remark}\rm
	\begin{enumerate}
		\item The third component of the Laplace-Runge-Lenz vector, $A_3$, also defines a Morse function on $\mathcal{M}_E$.
		Its critical points are the same as those of $L_3$, but the indices differ:
		$\gamma_{c_-}$ is the minimum, $\gamma_{c_+}$ is the maximum, and $\gamma_\pm$ are saddles.
		
		\item We can regard $A_3$ also as a Morse function on each regular level set $L_3^{-1}(a)$.
		In this case, the level sets of $A_3$ are tori $T^2$ except at two extrema, where they degenerate to $S^1$.
		The singular level set $L_3^{-1}(0)$ has $\gamma_{c_\pm}$ as isolated extrema and is homeomorphic to the suspension of $T^2$, which is not a manifold.
		
		\item These observations suggest a toric-like structure on $\mathcal{M}_E$, as illustrated in \Cref{Figure : Toric figure}.
		One may view $(L_3, A_3)$ as a moment map for a $T^2$-action on $\mathcal{M}_E$,
		where the action corresponding to $L_3$ arises from rotations about the $q_3$-axis, and that for $A_3$ reflects a \emph{hidden symmetry}.
	\end{enumerate}
\end{remark}

\begin{figure}[ht]
	\centering
	\includegraphics[width=9cm]{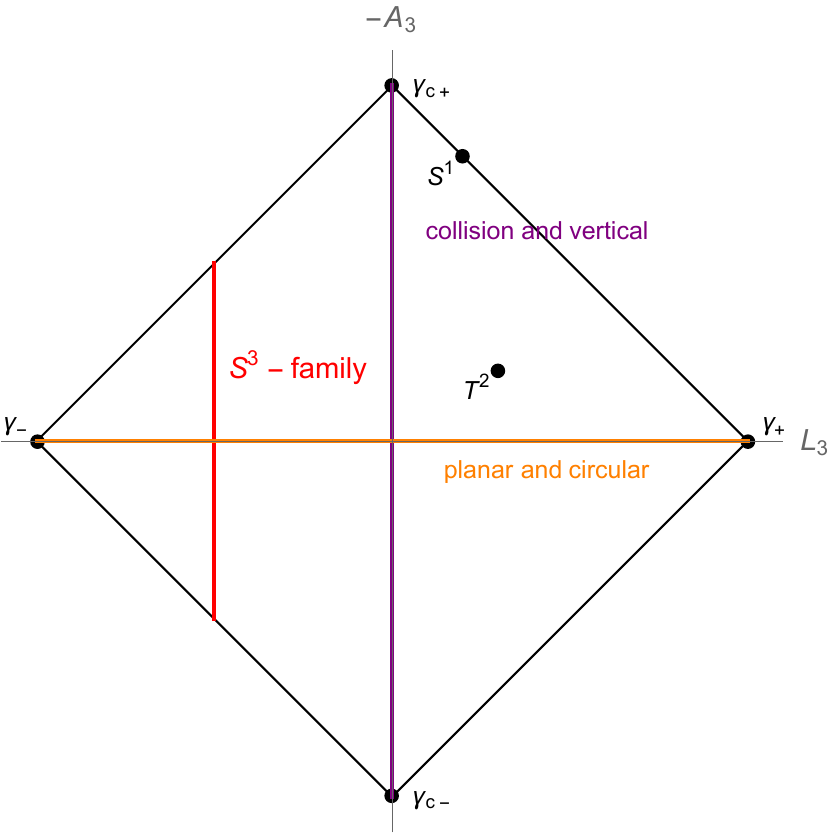}
	\caption{Toric-style diagram of the moduli space $\mathcal{M}_E$.}
	\label{Figure : Toric figure}
\end{figure}

\begin{remark}\rm
	There is a kind of duality between $L_3$ and $A_3$.
	If we define the modified Hamiltonian
	$$
	H_A = E + A_3,
	$$
	then $A_3$ plays the role of a new energy function on $\mathcal{M}_E$, with $\gamma_{c_\pm}$ as extrema and $\gamma_\pm$ as saddles.
	This suggests an alternative Kepler-type system with the same non-degenerate orbits but different Morse-Bott families, still diffeomorphic to $S^3$.
	Here, circular and collision orbits switch roles, and the hidden symmetry becomes more prominent.
\end{remark}

\Cref{Figure : Bifurcation diagram} shows a bifurcation of the Kepler orbits.
As $c$ increases, a family $\Sigma_{k,l}$ is born at the $(k-l)$-th iterate of the direct orbit and dies at the $(k+l)$-th iterate of the retrograde orbit.
The bifurcation pattern closely resembles that in the planar case discussed in \cite{Albers_Fish_Frauenfelder_vanKoert_13}.
However, since $L_3^{-1}(0)$ is not a manifold, this bifurcation has a singularity at $L_3=0$ and the families with positive and negative $L_3$ should be treated separately.

\begin{figure}[ht]
	\centering
	\includegraphics[width=7cm]{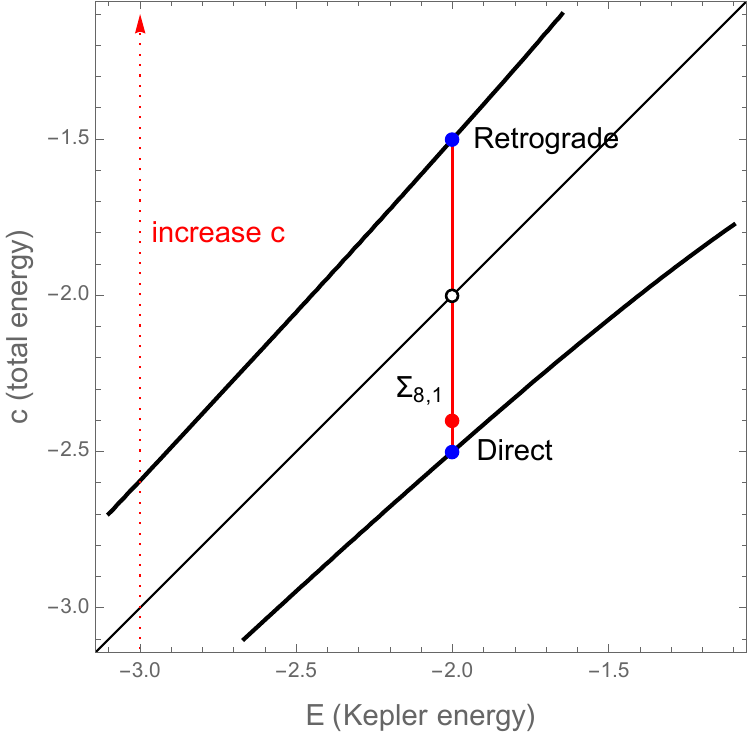}
	\caption{Bifurcation diagram at $E_{8,1}$.}
	\label{Figure : Bifurcation diagram}
\end{figure}

\section{Indices of Periodic Orbits}\label{Section - CZ index of periodic orbits}

\subsection{Conley-Zehnder Index and Robbin-Salamon Index}\label{Subsection - CZ Index}

Let $Q$ be a quadratic form defined on the vector space $V$.
In an appropriate basis, $Q$ can be represented as a diagonal matrix.
Let $n_+$, $n_0$ and $n_-$ denote the counts of positive, zero and negative entries respectively.
The \textbf{signature} of $Q$ is defined as
$$
\Sign(Q) = n_+ - n_-.
$$

Let $\Psi:[0,\tau]\to Sp(2n)$ be a path of symplectic matrices with $\Psi(0)=\Id$. A point $t\in [0,\tau]$ is called a \textbf{crossing} if $\ker(\Psi(t)-\Id)\neq0$. The \textbf{crossing form} is a quadratic form defined on the vector space $V_t = \ker(\Psi(t)-\Id)$ as follows
$$
Q_t(v,v) = \o(v,\dot{\Psi}(t)v),
$$
where $\o$ is a symplectic form on $V_t$.
Using a symplectic basis $\{v_1,w_1,\cdots,v_n,w_n\}$ where $\o(v_i,w_j)=\delta_{ij}$, the crossing form can be expressed as
$$
Q_t = \Omega\dot{\psi}(t) = 
\mathrm{diag}\left(
\begin{pmatrix}
	0&1\\
	-1&0
\end{pmatrix},
\cdots,
\begin{pmatrix}
	0&1\\
	-1&0
\end{pmatrix}
\right)
\dot{\psi}(t).
$$
Assume the crossings are regular, which means that the crossing forms are non-degenerate.
It follows that the crossings are isolated.
The \textbf{Robbin-Salamon index}, introduced in \cite{Robbin_Salamon_93}, for $\Psi$ is given by
$$
\mu_{RS}(\psi) = \frac{1}{2}\Sign(Q_0) + \sum_t \Sign(Q_t) + \frac{1}{2}\Sign(Q_\tau),
$$
where $\sum_t$ is the sum over all interior crossings.
The Robbin-Salamon index has the following key properties.
\begin{thm}\label{Theorem - RS index axiom}
	Let $\Psi_i:[0,\tau]\to Sp(2n)$ be paths of symplectic matrices.
	The Robbin-Salamon index $\mu_{RS}$ satisfies
	\begin{enumerate}
		\item (\rm{Homotopy Invariance}) If $\Psi_1$ and $\Psi_2$ are homotopic,
		$$\mu_{RS}(\Psi_1) = \mu_{RS}(\Psi_2).$$
		\item (\rm{Additivity}) Let $\Psi_3(t)=\Psi_2(t)\Psi_1(t)$ be the pointwise product of $\Psi_1$ and a loop $\Psi_2$.
		Then,
		$$
		\mu_{RS}(\Psi_3) = \mu_{RS}(\Psi_2)+\mu_{RS}(\Psi_1).
		$$
	\end{enumerate}
\end{thm}
\begin{proof}
	For proofs and additional details, see \cite{Salamon_Zehnder_92}, \cite{Robbin_Salamon_93}, or \cite{Salamon_99}.
	Note that the Robbin-Salamon index of a loop may differ from the Maslov index by a factor of two, depending on conventions.
\end{proof}

Let $(Y^{2n+1},\ker\alpha)$ be a contact manifold, $R$ be the Reeb vector field and $\g:[0,\tau]\to Y$ be a periodic Reeb orbit.
Assume that $\g$ is contractible, so there exists a capping disk $D_\g:D^2\to Y$ such that $D_\g|_{\pp D^2} = \g$.
Choosing a trivialization of the contact structure $\mathcal{T}:\g^*\xi\to[0,\tau]\times \R^{2n}$ along $\g$, which can be extended to the capping disk as a trivialization $\tilde{\mathcal{T}}:D_\g^*\xi\to D_\g\times\R^{2n}$.
We obtain a path of symplectic matrices
$$\Psi_\g(t) = \mathcal{T}(t) dFl^R_t|_\xi \mathcal{T}(0)^{-1}\in Sp(2n).$$
The \textbf{Conley-Zehnder index} of $\g$ is defined as
$$\mu_{CZ}(\g) = \mu_{RS}(\Psi_\g).$$
By the homotopy invariance property in \Cref{Theorem - RS index axiom}, this definition does not depend on the choice of trivialization.
\begin{remark}\rm
	\begin{enumerate}
		\item The Conley-Zehnder index, originally defined in \cite{Conley_Zehnder_83}, is defined for paths such that $\Psi_\g(\tau)$ does not have 1 as an eigenvalue, which means the orbit is non-degnerate.
		However, the Robbin-Salamon index can be defined for degenerate orbits.
		In this sense, we define the index of a degenerate orbit $\g$ as $\mu_{RS}(\Psi_\g)$ and call it the \textbf{Robbin-Salamon index} of $\g$.
		\item In general, the Conley-Zehnder index depends on the choice of the capping disk and the difference can be described in terms of the first Chern class $c_1$ of the symplectic manifold.
		In our case, the target manifold is $T^*S^3$ of which the first Chern class is zero, so there is no dependency on the capping disk.
	\end{enumerate}
\end{remark}

Let $H$ be a Hamiltonian on a symplectic manifold $W^{2n}$, and $\g:[0,\tau]\to W$ be a periodic Hamiltonian orbit on a regular level set $H^{-1}(c)$.
Assume that $Y=H^{-1}(c)$ is of a contact type, so that there exists a Liouville vector field $X$ and $(Y,\ker i_X\o|_Y)$ is a contact manifold and the Hamiltonian vector field $X_H$ on $Y$ is parallel to the Reeb vector field $R$ of $Y$.
Once we choose a symplectic trivialization along the orbit, we can view the linearized flow as a path of symplectic matrices and define the index.
There are two possible definitions of Conley-Zehnder index of a Hamiltonian orbit $\g$.
The first one is defined by the linearized Hamiltonian flow
$$dFl^{X_H}:[0,\tau]\to Sp(2n),$$
and the second one is defined by the linearized Hamiltonian flow restricted to the contact structure of the level set $H^{-1}(c)$,
$$
dFl^{X_H}|_{\xi}:[0,\tau]\to Sp(2n-2).
$$
In this paper, we take the second one, the transverse Conley-Zehnder index, as a definition of the Conley-Zehnder index of a Hamiltonian orbit.

\begin{remark}\rm
	To distinguish these two Conley-Zehnder indices, the second one is usually called a \textbf{transverse Conley-Zehnder index}.
\end{remark}

Since the Hamiltonian orbit is a reparametrization of the Reeb orbit, we consider the effect of a reparametrization on the Conley-Zehnder index.
Let $\tilde{\g}(s(t))=\g(t)$, where $s(0)=0$, $s(\tau)=\sigma$.
Denote the initial and final points as $\g(0)=\tilde{g}(0)=q_0$ and $\g(\tau)=\tilde{\g}(\sigma)=q_1$.
The linearized flows satisfy
$$
\begin{aligned}
	dFl^{X_H}_\sigma&:T_{q_0} W \to T_{q_1} W,\\
	dFl^R_\tau&:T_{q_0}Y\to T_{q_1}Y.
\end{aligned}
$$
\begin{lmm}\label{Lemma - Reparametrization and linearized flow}
	Let $X$ be a vector field on a manifold $W$, and $s:W\to \R$.
	Define $\psi(q) = Fl^X_{s(q)}(q)$.
	Then,
	$$
	d\psi(q) \xi = dFl^X_{s(q)}(q)\xi + (ds(q)\xi)X.
	$$
\end{lmm}
\begin{proof}
	See Section 9.1 of \cite{Frauenfelder_van_Koert_18}.
\end{proof}
Let $N$ be a normal vector to $Y$ in $W$.
Then we have
$$T_q W = \langle N_q \rangle \oplus T_q Y = \langle N_q \rangle \oplus \langle R_q \rangle \oplus \xi_q.$$
Consider the map given by quotient
$$
dFl^{X_H}_\sigma:T_{q_0} Y \to T_{q_1} Y.
$$
From \Cref{Lemma - Reparametrization and linearized flow}, the difference between $dFl^{X_H}_\sigma$ and $dFl^R_\tau$ is parallel to $R$.
This implies that, after quotienting by $\langle R\rangle$, the two maps
$$
\begin{aligned}
	dFl^{X_H}_\sigma&:\xi_{q_0}\to\xi_{q_1},\\
	dFl^R_\tau&:\xi_{q_0}\to\xi_{q_1},
\end{aligned}
$$
are identical.
We can summarize this result as follows.
\begin{prop}\label{Proposition - CZ index can be computed by Hamiltonian}
	Let $H:(M,\o)\to \R$ be a Hamiltonian, and let $Y=H^{-1}(c)$ be a regular level set of contact type.
	Let $\g$ be a Reeb orbit, and let $\tilde{\g}(s(t))=\g(t)$ be its corresponding Hamiltonian orbit. Then,
	$$
	dFl^R_t|_\xi = dFl^{X_H}_{s(t)}|_\xi.
	$$
	In particular, the Conley-Zehnder index can be computed using the linearized Hamiltonian flow.
\end{prop}

\begin{remark}\rm
	Let $H$ be a Hamiltonian on a symplectic manifold $W^{2n}$, and $\g:[0,\tau]\to W$ be a periodic Hamiltonian orbit.
	Then there are two possible definitions of Conley-Zehnder index.
	The first one is defined by the linearized Hamiltonian flow
		$$dFl^{X_H}:[0,\tau]\to Sp(2n),$$
	and the second one is defined by the linearized Hamiltonian flow restricted to the contact structure of the level set $H^{-1}(c)$,
		$$
		dFl^{X_H}|_{\xi}:[0,\tau]\to Sp(2n-2).
		$$
	To distinguish these, the second one is usually called a \textbf{transverse Conley-Zehnder index}.
	In this paper, we take the second one, the transverse Conley-Zehnder index, as a definition of the Conley-Zehnder index of a Hamiltonian orbit.
\end{remark}

\subsection{Indices of the Planar Circular Orbits}\label{Subsection - Planar Index}

The computation of the indices of the retrograde and direct orbits in the planar problem is described in \cite{Albers_Fish_Frauenfelder_vanKoert_13}.
Basically the strategy of computation is the same, except that we're using cylindrical coordinates.
The parametrization and the period of the planar circular orbits in terms of cylindrical coordinates are given in \Cref{Lemma - Periods of Retrograde and Direct},
	$$
	\g_\pm(t) = \begin{pmatrix}
		\o_0^2\\
		(1/\o_0^3 +1)t\\
		0\\
		0\\
		\o_0\\
		0
	\end{pmatrix},\quad
	\o_0 = \pm \frac{1}{\sqrt{-2E}}.
	$$
From \Cref{Proposition - CZ index can be computed by Hamiltonian}, the Conley-Zehnder index can be computed using the linearized Hamiltonian flow.
Let $\mathbf{L}$ be the linearization matrix obtained by differentiating $X_H$,
$$
\mathbf{L}
=
\begin{pmatrix}
	0&0&0&1&0&0\\
	-2p_\theta/r^3&0&0&0&1/r^2&0\\
	0&0&0&0&0&1\\
	-3p_\theta^2/r^4 + 3r^2/R^5-1/R^3&0&3rz/R^5&0&2p_\theta/r^2&0\\
	0&0&0&0&0&0\\
	3rz/R^5&0&3z^2/R^5-1/R^3&0&0&0
\end{pmatrix},
\quad
R = \sqrt{r^2+z^2}.$$
After substituting the orbit $\g_\pm$, we find
$$
\mathbf{L}_\pm =\begin{pmatrix}
	0   &0  &0  &1  &0  &0\\
	- 2/\o_0^5   &0   &0 &0  &1/\o_0^4   &0\\
	0   &0  &0  &0  &0  &1\\
	-1/\o_0^6   &0  &0  &0  & 2/\o_0^3&  0\\
	0   &0  &0  &0  &0  &0\\
	0   &0  &-1/\o_0^6  &0  &0  &0
\end{pmatrix}.
$$
We need to determine a symplectic frame of $\ker(dH)\cap \ker(\lambda)$ along the orbit.
We have
$$
	dH =\left(-\frac{p_\theta^2}{r^3} + \frac{r}{(r^2+z^2)^{3/2}}\right)dr +\frac{z}{(r^2+z^2)^{3/2}}dz +p_r dp_r + \left(1+\frac{p_\theta}{r^2}\right)dp_\theta + p_z dp_z,
$$
and the contact form is given by
$$\lambda = -qdp = p_\theta d\theta-rdp_r - zdp_z.$$
These simplify to
$$
dH =\left(1+ \frac{1}{\o_0^3}\right)dp_\theta,\quad \lambda = \o_0 d\theta - \o_0^2 dp_r,
$$
along the orbits $\g_\pm$.

\begin{note}\rm
	Here, we use $-qdp$ instead of the standard $pdq$, as in the context of Moser regularization, the roles of $p$ and $q$ are interchanged.
	Specifically, let $X=q\pp_q$ so that $i_X \o = -qdp$.
	Note that $X(H) = 1/|q|+L_3$, and $X(H)$ must be positive.
	This can be seen as a special case of a similar phenomenon in the restricted three-body problem.
	See Theorem 5.2.1 in \cite{Cho_Jung_Kim_20}.
\end{note}

With this, we can find a symplectic frame
$$
(X_1, X_2, X_3, X_4)= \left( \pp_\theta + \frac{1}{\o_0}\pp_{p_r}, \o_0 \pp_r, \pp_{p_z},\pp_z\right),
$$
where $\o(X_1,X_2)=\o(X_3,X_4)=1$ and $\o(X_i,X_j)=0$.

\begin{lmm}[\cite{Albers_Fish_Frauenfelder_vanKoert_13}, Appendix B]\label{Lem - Framing in S2}
	Let $T^*S^2\subset T^*\R^3$ with coordinates $(x,y)$ such that $|x|^2=1$ and $ x \cdot y=0$.
	Let $K:T^*S^2\to\R$ be a fiberwise star-shaped Hamiltonian such that $y\pp_y K>0$, and let the contact form be given by $\ld = ydx$.
	Then the following vector fields provide a global trivialization of $\ker \ld \cap \ker dK$.
	$$
	\begin{aligned}
		X_1 &= \left( y\times x - \frac{ (y\times x)\cdot \pp_y K}{y\cdot \pp_y K}\right)\cdot \pp_y,\\
		X_2 &=-\frac{(y\times x)\cdot \pp_x K}{y\cdot  \pp_y K}y\cdot \pp_y + (y\times x)\cdot \pp_x.
	\end{aligned}
	$$
\end{lmm}

\begin{lmm}\label{Lem - Extending planar frame}
	The framing $X_1,X_2,X_3,X_4$ given above can be extended to a capping disk.
\end{lmm}
\begin{proof}
	From \cite{Albers_Fish_Frauenfelder_vanKoert_13} Appendix B, we can see that $X_1,X_2$ correspond to the framing in the planar orbit given in \Cref{Lem - Framing in S2}.
	Thus, $X_1,X_2$ can be extended to a planar capping disk.
	Since $X_3,X_4$ do not involve planar coordinates, they can also be extended to the planar capping disk if the capping disk is push away from the collision locus.
	This can be done by, for example, using a cutoff function to push the disk into $q_3$- and $p_3$-directions slightly.
\end{proof}

For a complete computation, we also take a normal frame
$$
(N_1, N_2)=\left( \frac{1}{\o_0}\pp_{\theta},\o_0\pp_{p_\theta}+\o_0^2\pp_r\right),
$$
such that $\o(N_1,N_2)=1$ and $\o(X_i,N_j)=0$.
Note that $N_1$ represents the Reeb direction of $\lambda$.
We have
$$
\begin{aligned}
	\mathbf{L}X_1 &= (1/\o_0)\pp_r = (1/\o_0^2)X_2,\\
	\mathbf{L}X_2 &= -(2/\o_0^4) \pp_\theta - (1/\o_0^5)\pp_{p_r}
	=-(1/\o_0^4)X_1 - (1/\o_0^3)N_1,\\
	\mathbf{L}X_3 &= \pp_z = X_4,\\
	\mathbf{L}X_4 &= -(1/\o_0^6)\pp_{p_z}=-(1/\o_0^6)X_3.
\end{aligned}
$$
The linearized flow matrix under the frame $(X_1,X_2,X_3,X_4)$ is
$$
\mathbf{L}=\begin{pmatrix}
	0&-1/\o_0^4&0&0\\
	1/\o_0^2&0&0&0\\
	0&0&0&-1/\o_0^6\\
	0&0&1&0
\end{pmatrix}.
$$
By taking the matrix exponential, we obtain the path of symplectic matrices
$$
\Psi_H(t)=\begin{pmatrix}
	\cos(t/\o_0^3)&-(1/\o_0)\sin(t/\o_0^3)&0&0\\
	\o_0\sin(t/\o_0^3)&\cos(t/\o_0^3)&0&0\\
	0&0&\cos(t/\o_0^3)&-(1/\o_0^3)\sin(t/\o_0^3)\\
	0&0&\o_0^3\sin(t/\o_0^3)&\cos(t/\o_0^3)
\end{pmatrix}.
$$
The crossings occur at $2\pi\o_0^3\Z$.
The crossing form is given by $\Omega \dot{\Psi}(t)$ where $\Omega$ is defined in \Cref{Subsection - CZ Index},
$$
\Omega = \begin{pmatrix}
	0&1&0&0\\
	-1&0&0&0\\
	0&0&0&1\\
	0&0&-1&0
\end{pmatrix}.
$$
We have
$$
\Omega\dot{\Psi}_H(t)=\Omega\mathbf{L} = \begin{pmatrix}
	1/\o_0^2&0&0&0\\
	0&1/\o_0^4&0&0\\
	0&0&1&0\\
	0&0&0&1/\o_0^6
\end{pmatrix},
$$
so the signature of the crossing form is always 4.
Using $\o_0=\pm1/\sqrt{-2E}$, we derive the following result.
\begin{thm}\label{Theorem - Index of planar circular orbits}
	Let $\g_\pm$ denote the retrograde and direct orbits of Kepler energy $E$, where $E\neq E_{k,l}$ for any $k,l$.
	Then, $\g_{\pm}$ and their multiple covers are non-degenerate.
	The Conley-Zehnder index of $N$-th iterate of $\g_\pm$ is
	$$
	\begin{aligned}
		\mu_{CZ}(\g_\pm^N)
		&= 2 + 4\max\left\{n\in \N:\frac{2\pi n}{(-2E)^{3/2}}<N\frac{2\pi}{(-2E)^{3/2}\pm 1}\right\}\\
		&=2 + 4\max\left\{n\in\N:n<N\frac{(-2E)^{3/2}}{(-2E)^{3/2}\pm1}
		\right\}\\
		&=2+4\left\lfloor N\frac{(-2E)^{3/2}}{(-2E)^{3/2}\pm1}\right\rfloor.
	\end{aligned}
	$$
\end{thm}
\begin{note}\label{Note - Good Orbit}\rm
	This is exactly twice the index of the retrograde and direct orbits in the planar problem, which was computed in \cite{Albers_Fish_Frauenfelder_vanKoert_13}.
\end{note}
We write
$$\mu_\pm=\mu_\pm(E)=\frac{(-2E)^{3/2}}{(-2E)^{3/2}\pm1}.$$
The graph of $\mu_\pm$ is illustrated in \Cref{Figure : mu graph}.

\begin{figure}[ht]
	\centering
	\includegraphics[width=7cm]{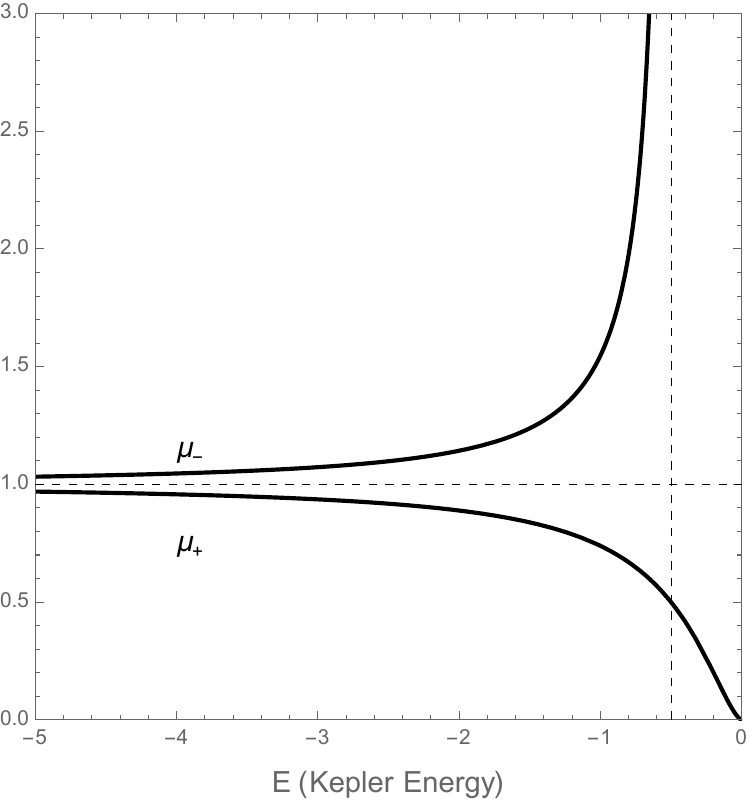}
	\caption{Graphs of $\mu_\pm$.}
	\label{Figure : mu graph}
\end{figure}

Here's some observation about the indices of $\g_{\pm}$.
Let the covering number $N$ be given.
\begin{enumerate}
	\item We can take a sufficiently small $E$ such that $N-1<N\mu_+<N<N\mu_-<N+1$.
	Hence,
	$$
		\mu_{CZ}(\g_+^N) = 4N-2,\quad
		\mu_{CZ}(\g_-^N) = 4N+2.
	$$
	\item The index of $N$-th cover of the retrograde orbit, $\mu_{CZ}(\g_+^N)$, decreases by 4 each time $\mu_+$ touches $k/N$ for some $k=1,\cdots,N-1$.
	The corresponding energy satisfies
	$$
	\frac{(-2E)^{3/2}}{(-2E)^{3/2}+1}=\frac{k}{N},
	$$
	or equivalently
	$$
	E  = E_{k,N-k}= -\frac{1}{2}\left(\frac{k}{N-k}\right)^{2/3}.
	$$
	\item The index of $N$-th cover of the direct orbit, $\mu_{CZ}(\mu_-^N)$, increases by 4 each time $\mu_-$ touches $1+k/N$ for some $k=1,2,\cdots$.
	The corresponding energy satisfies
	$$
	E = E_{N+k,k} = -\frac{1}{2}\left(\frac{N+k}{k}\right)^{2/3}.
	$$
\end{enumerate}

We summarize the result as follows.
\begin{cor}\label{Corollary - Index of Planar Orbits}
	For fixed $N$, the Conley-Zehnder index of $N$-th cover of retrograde orbit with Kepler energy $E$ is
	$$
	\mu_{CZ}(\g_+^N)=\left\{\begin{array}{cc}
		4N-2&   \text{if }E<E_{N-1,1},\\
		4(N-k)-2&\begin{array}{c}\text{if }E_{N-k,k}<E<E_{N-k-1,k+1},\\
			\text{for }k=1,2,\cdots,N-2\end{array}\\
		2 &\text{if }E>E_{1,N-1}.
	\end{array}
	\right.
	$$
	In particular, the simple retrograde orbit has always index 2.
	
	Similarly, the Conley-Zehnder index of $N$-th cover of the direct orbit is
	$$
	\mu_{CZ}(\g_-^N)=\left\{\begin{array}{cc}
		4N+2&   \text{if }E<E_{N+1,1},\\
		4(N+k)+2&   \begin{array}{c}\text{if }E_{N+k,k}<E<E_{N+k+1,k+1},\\
			\text{for }k=1,2,\cdots.\end{array}
	\end{array}
	\right.
	$$
	In particular, the indices of direct orbits diverge as $E\to -1/2$.
\end{cor}


\subsection{Indices of Vertical Collision Orbits}\label{Subsection - Vertical index}

We first clarify the notations which will be used in this subsection to avoid confusion.
We fix the Jacobi energy $c$.
\begin{itemize}
	\item $r=\sqrt{-2E}$ : The radius of the sphere used in the Moser regularization.
	After scaling, the domain of the Hamiltonians are the cotangent bundle of unit sphere, while $r$ is encoded in the length of cotangent vector $y$; $|y|=1/r$.
	We will always regard the scaling is applied to the regularized Hamiltonians, of which the explicit formula is given in \Cref{Subsection - Moser regularization}.
	\item $E$ : the Kepler energy.
	\item $K_E$ : the Hamiltonian of Kepler problem under the Moser regularization.
	\item $H = E+L_3$ : the Jacobi energy.
	\item $K_H$ : the Hamiltonian of rotating Kepler problem under the Moser regularization.
\end{itemize}

We've seen in \Cref{Subsection - Classification of periodic orbits} that vertical collision orbits are parametrized by
$$
	\g_{c_\pm}(t) = (x(t);y(t)) = (-\cos (rt), 0, 0, \pm\sin(rt); (1/r)\sin(rt),0,0,\pm(1/r)\cos(rt))
$$
in $T^*S^3$, where $r=\sqrt{-2E}=\sqrt{-2c}$.
Note that $\g_{c_\pm}$ can be regarded as both $K_H$-orbit and $K_E$-orbit.

\begin{lmm}\label{Lemma - Vertical Frame}
	Along the vertical collision orbits $\g_{c_\pm}$ in $T^*S^3$, we can take
	$$
	(X_1,X_2,X_3,X_4)=(\pp_{y_1},\pp_{x_1},\pp_{y_2},\pp_{x_2})
	$$
	as a symplectic frame of the contact structure of $K_H^{-1}(1/2)$ and $K_E^{-1}(1/2)$.
\end{lmm}
\begin{proof}
	We can take $y\cdot \pp_y$ as a Liouville vector field, so $\lambda=i_X dy\w dx=ydx$ as a Liouville form whose restriction to the energy hypersurface is a contact form.
	It is clear that $X_i$ are tangent to $T^*S^3$ along the vertical collision orbit.
	Additionally,
	$$
	\begin{aligned}
		dK_H &= r^2 y_0 dy_0 + r^2 y_3 dy_3,\\
		\lambda &= y_0 dx_0 + y_3 dx_3,
	\end{aligned}
	$$
	along the orbit, so $X_i\in \ker \ld \cap \ker dK_H$.
	Also $\o(X_1,X_2)=\o(X_3,X_4)=1$ and $\o(X_i,X_j)=0$ otherwise, verifying that $(X_1,X_2,X_3,X_4)$ forms a symplectic frame.
	Since $dK_E=dK_H$ along $\g_{c_\pm}$, the frame is still valid if we regard $\g_{c_\pm}$ as a $K_E$-orbit.
\end{proof}
\begin{lmm}\label{Lemma - Vertical Frame 2}
	The frame given in \Cref{Lemma - Vertical Frame} can be extended to a symplectic frame on a capping disk of $\g_{c_\pm}$ in $K_E^{-1}(1/2)$.
\end{lmm}
\begin{proof}
	For simplicity, assume $r=1$.
	Consider the cotangent bundle of the subspace
	$$
	S = \{(\cos t \cos \theta,\sin\theta,0,\sin t\cos\theta)\}\subset S^3.
	$$
	Here, $T^*S$ is a subset of $T^*\R^3 = \{x_2=y_2=0\}$.
	Applying \Cref{Lem - Framing in S2} to the $(x_0,x_1,x_3)$-coordinate system and the Hamiltonian $K=|y|^2/2$, we observe that $X_1,X_2$ match the frame given in the lemma along the orbit.
	If we take a capping disk $D$ in $ST^*S=K^{-1}(1/2)\cap T^*S$, as in \Cref{Lem - Extending planar frame}, $X_3,X_4$ can be extended to $D$.
\end{proof}

One can check that the linearized Hamiltonian flow of $K_H$ is time-dependent, and directly integrating it is quite challenging.
Instead, we use the following lemma to compute the Conley-Zehnder index.
\begin{lmm}\label{Lemma - Decomposition of Index}
	Let $\Psi_{K_E}$ and $\Psi_{L_3}$ be paths of symplectic matrices given by the linearized flow of $K_E$ and $L_3$.
	Then,
	$$
	\mu_{CZ}(\g_c) = \mu_{RS}(\Psi_{K_E}) + \mu_{RS}(\Psi_{L_3}).
	$$
\end{lmm}
\begin{proof}
	See the end of this subsection.
\end{proof}

The Hamiltonian flow of $K_E$, which was given in \Cref{Subsection - Moser regularization}, and its linearized flow is
$$
\begin{pmatrix}
	\dot{x}_0\\
	\dot{x}_3\\
	\dot{y}_0\\
	\dot{y}_3
\end{pmatrix}
=
\begin{pmatrix}
	r^2 y_0\\
	r^2 y_3\\
	- x_0\\
	-x_3
\end{pmatrix},
\qquad
\mathbf{L} = \begin{pmatrix}
	0&-1&0&0\\
	r^2&0&0&0\\
	0&0&0&-1\\
	0&0&r^2&0
\end{pmatrix}.
$$
After integration, the resulting path of symplectic matrices is
$$
\Psi_{K_E}(t)=\begin{pmatrix}
	\cos(rt)&\sin(rt)/r&0&0\\
	-r\sin(rt)&\cos(rt)&0&0\\
	0&0&\cos(rt)&\sin(rt)/r\\
	0&0&-r\sin(rt)&\cos(rt)
\end{pmatrix}.
$$
Notice that the period of $\Psi_{K_E}$ is equal to the period of the collision orbit.
The crossing occurs at $t=(2\pi/r)\Z$, and the crossing form is
$$
\Omega\dot{\Psi}_{K_E}(\tau) = \begin{pmatrix}
	r^2&0&0&0\\
	0&1&0&0\\
	0&0&r^2&0\\
	0&0&0&1
\end{pmatrix}.
$$
It follows that the signature of the crossing form is always 4.

Next, we compute the index of the orbit of the angular momentum $L_3$ on $T^*\R^3$.
The Hamiltonian equation and the linearized matrix $\mathbf{M}$ with respect to the symplectic basis $\pp_{p_1},\pp_{q_1},\pp_{p_2},\pp_{q_2}$ is given by
$$
\begin{pmatrix}
	\dot{q}_1\\
	\dot{q}_2\\
	\dot{q}_3\\
	\dot{p}_1\\
	\dot{p}_2\\
	\dot{p}_3
\end{pmatrix}
=
\begin{pmatrix}
	-q_2\\
	q_1\\
	0\\
	-p_2\\
	p_1\\
	0
\end{pmatrix},
\qquad
\mathbf{M} = \begin{pmatrix}
	0&0&1&0\\
	0&0&0&1\\
	-1&0&0&0\\
	0&-1&0&0
\end{pmatrix}.
$$
The linearized flow is given by
$$
\Psi_L(t) = \begin{pmatrix}
	\cos t&0&\sin t&0\\
	0&\cos t&0&\sin t\\
	-\sin t&0&\cos t&0\\
	0&-\sin t&0&\textit{}\cos t
\end{pmatrix}.
$$
The crossing occurs at $2\pi\Z$, and the crossing form is
$$
\Omega\dot{\Psi}_L(\tau) = \Omega\mathbf{M} = \begin{pmatrix}
	0&0&0&1\\
	0&0&-1&0\\
	0&-1&0&0\\
	1&0&0&0
\end{pmatrix},
$$
so the signature of the crossing form is 0.
Combining these results, we have the following.
\begin{thm}\label{Theorem - Index of Collision}
	Let $\g_{c_\pm}$ be the vertical collision orbits of Kepler energy $E$, where $E\neq E_{k,l}$ for any $k,l$.
	Then $\g_{c_\pm}$ and their multiple covers are non-degenerate.
	The Conley-Zehnder index of the $N$-th iterate of $\g_{c_\pm}$ is
	$$
	\begin{aligned}
		\mu_{CZ}(\g^N_{c_\pm})&=\mu_{RS}(\Psi_{K_E}) + \mu_{RS}(\Psi_L)\\
		&=\left(2+4(N-1\right) + 2)+ 0 = 4N.
	\end{aligned}
	$$
	In particular, the index of the multiple cover of vertical collision orbits never changes.
\end{thm}

\subsubsection*{Proof of \Cref{Lemma - Decomposition of Index}}\label{Subsubsection - Justifying}

Let $\g:[0,\tau]\to T^*S^3$ be a $K_H$-orbit.
Since the flow of $K_H$ and $H$ are parallel, with \Cref{Lemma - Reparametrization and linearized flow} we have
$$
\mu_{CZ}(\g) = \mu_{RS}(\Psi_{K_H}) = \mu_{RS}(\Psi_H)
$$
except for the collision orbits, for which $H$ is not defined.

We first extend $\Psi_H$ to the collision locus.
Let $\g:[0,\tau]\to T^*S^3$ be a $K_H$-orbit such that $\g(0)$ lies on the collision locus.
Let $\mathcal{T}:\g^*\xi\to [0,\tau]\times \R^4$ be a trivialization of the contact structure along $\g$.
For $\eps>0$, we can find an $H$-orbit $\tilde{\g}_\eps:[0,\sigma(\eps)]\to T^*\R^3$, which is a reparametrization of $\g|_{[\eps,\tau]}$.
Under the inverse stereographic projection, we can regard $\tilde{\g}_\eps$ as an orbit on $T^* S^3$.
We define
$$\begin{aligned}
	\Phi^\eps_\tau &= \mathcal{T}(t)^{-1}dFl^{X_H}_{\sigma(\eps)}|_{\xi}\mathcal{T}(\eps) \in Sp(4),\\
	\Psi_\eps &= \mathcal{T}(\eps)^{-1}dFl^{X_K}_\eps|_\xi \mathcal{T}(0) \in Sp(4).
\end{aligned}
$$
Then we have $\Psi_\tau = \Phi^\eps_\tau \Psi_\eps$.
It's clear that $\lim_{\eps\to0}\Psi_\eps =\Psi_0=\Id$, so that
$$
\lim_{\eps\to0} \Phi^\eps_\tau = \Psi_\tau.
$$
In a similar way, we can extend $\Psi_E$ to the collision locus.

Now since $H=E+L$ and $\{E,L\}=0$, we have
$$
dFl^{X_H}_t = dFl^{X_{L_3}}_t \circ dFl^{X_E}_t.
$$
If we restrict ourselves to the collision orbit, we get
$$
\mathcal{T}(t)dFl^{X_H}_t|_\xi \mathcal{T}(0)^{-1} = \left( \mathcal{T}(t) dFl^{X_{L_3}}_t|_\xi \mathcal{T}(t)^{-1}\right) \left( \mathcal{T}(t) dFl^{X_E}_t|_\xi \mathcal{T}(0)^{-1}\right).
$$
From \Cref{Theorem - RS index axiom}, we know that
$$
\mu_{RS}(\Psi_H) = \mu_{RS}(\Psi_{L_3}) + \mu_{RS}(\Psi_E).
$$
Since $E$-orbit is a reparamerization of $K_E$-orbit, we have $\mu_{RS}(\Psi_E) = \mu_{RS}(\Psi_{K_E})$, which leads us to the conclusion.
Moreover, since the $L_3$-orbit with initial conditions on the collision orbit is constant, we can compute the index with respect to any frame.

\subsection{Relation with the Symplectic Homology}\label{Subsection - Relation with SH}

The \textbf{symplectic homology} is a homology defined on a Liouville manifold, given by a chain complex generated by periodic Hamiltonian orbits, which are graded by the Conley-Zehnder index, and the differential defined in terms of a Floer cylinders.
Equipped with the Jacobi energy $H:T^*S^3\to \R$, we can consider the symplectic homology of $T^*S^3$.
Since we'll only deal with the computed result, we refer \cite{Audin_Damien_14}, \cite{Abouzaid_15}, \cite{Kwon_van_Koert_16}, and \cite{Gutt_18} for the detailed definition and further discussions.

The symplectic homology of $T^*S^3$ can be computed independently by Viterbo's theorem or other methods.
For the simplicity of the argument, we use the $S^1$-equivariant symplectic homology.
The derivation of the following result can be found in \cite{Kwon_van_Koert_16}, Proposition 5.12.

\begin{prop}\label{Proposition - S1 equivariant homology of TS3}
	The $+$-part of the $S^1$-equivariant symplectic homology of $T^*S^3$ is given by
	$$
	SH^{S^1,+}_* (T^*S^3;\Q) = \left\{\begin{array}{cc}
		\Q&\text{if}\,\,*=2,\\
		\Q^2&\text{if}\,\,*=2k\geq4,\\
		0&\text{otherwise}.
	\end{array}\right.
	$$
\end{prop}

We now relate this result to our computation of the Conley-Zehnder indices of the non-degenerate orbits.
Fix $N\in\mathbb{N}$.
We denote
$$
c^\pm_{k,l} = E_{k,l}\pm \frac{1}{\sqrt{-2E_{k,l}}}
$$
for the Jacobi energy of the retrograde orbit and direct orbit with Kepler energy $E_{k,l}$.
There exists an energy level $c\ll -3/2$ sufficiently small such that
\begin{enumerate}
	\item $\mu_{CZ}(\g_+^k)=4k-2$ for $k\leq N$,
	\item $\mu_{CZ}(\g_-^k)=4k+2$ for $k\leq N$,
	\item Higher iterates of the retrograde and direct orbits have index $\geq 4N-2$.
\end{enumerate}
This condition is achieved by taking $c$ smaller than $c^+_{N,1}$, so that the $k\leq N$-multiple covers of retrograde and direct orbits does not have index change.
Each cover of $\g_\pm$ and $\g_{c_\pm}$ is a Morse-Bott manifold $S^1$ since it is isolated, and generates $SH^{S^1,+}(T^*S^3;\Q)$.
Analyzing this setup, we find
\begin{enumerate}
	\item One generator at degree $2$ ($\g_+$),
	\item Two generators at degree $4k+2$ for $k=1,\ldots,N-1$ ($\g_+^{k+1}$ and $\g_-^k$),
	\item Two generators at degree $4k$ for every $k\in\mathbb{N}$ ($\g_{c_\pm}^k$).
\end{enumerate}
Since the degree gap between generators are $2$, differentials can be neglected.
This result coincides with the $+$-part of $S^1$-equivariant symplectic homology of $T^*S^3$ described in \Cref{Proposition - S1 equivariant homology of TS3} up to degree $4N-2$.
In other words, we can compute $SH^{S^1,+}(T^*S^3;\Q)$ up to desired degree in terms of non-degenerate orbits of the rotating Kepler problem.

\subsection{Morse-Bott Property of $\Sigma_{k,l}$-families}\label{Subsection - Morse-Bott property}

Following notions are taken from \cite{Kwon_van_Koert_16}.
Let $H$ be a Hamiltonian defined on a symplectic manifold $W$.
We say a family of $\tau$-periodic Hamiltonian orbits is a \textbf{Morse-Bott family} if following two conditions are satisfied.
	\begin{enumerate}
		\item $C=\{x\in W\,:\,Fl^{X_H}_\tau(x)=x\}$ forms a compact submanifold without boundary.
		\item Let $\Sigma$ be a connected component of $C$ and $\nu(\Sigma)$ be a normal bundle.
		The linear map
			$$
			d_x Fl^{X_H}_\tau|_{\nu(\Sigma)}-\Id|_{\nu(\Sigma)}
			$$
		is invertible for any $x\in \Sigma$.
	\end{enumerate}
Our target is $\Sigma_{k,l}$-family, which are the degenerate orbits.
To show the Morse-Bott property, we need to compute the linearized flow.
For this, we need the help of action-angle coordinates, which is given by the celebrated Arnold-Liouville theorem.
The proof can be found in \cite{Arnold_89}.

Let $H$ be a Hamiltonian defined on a symplectic manifold $(W^{2n},\o)$.
The system $(W,\o,H)$ is called \textbf{integrable} if there exists integrals of motion $f_1=H$, $f_2$, $\cdots$, $f_n$ such that $\{f_i,f_j\}=0$ and the differentials $(df_1)_p,\cdots,(df_n)_p$ are linearly independent.

\begin{thm}[Arnold-Liouville]
	Let $(W,\o,H)$ be an integrable system with integrals $f_1,\cdots,f_n$ and $c\in\R^n$ be a regular value of $F=(f_1,\cdots,f_n):W\to\R^n$.
	\begin{enumerate}
		\item The regular level set $F^{-1}(c)$ is a Lagrangian submanifold.
		\item If the Hamiltonian flows $Fl^{X_{f_1}}$, $\cdots$, $Fl^{X_{f_n}}$ starting at any $p\in F^{-1}(c)$ are complete, then each connected component of $F^{-1}(c)$ is a homogeneous space for $\R^n$, i.e., $\R^{n-k}\times T^k$.
		\item There exists a coordinate system $(\vp_1,\cdots,\vp_n)$ on each component of $F^{-1}(c)$ such that the Hamiltonian flows $Fl^{X_{f_i}}$ are linear.
		These are called \textbf{angle coordinates}.
		\item There exists symplectic conjugate coordinates $(p_1,\cdots,p_n)$ such that $p_i$'s are integrals, and coordinates $(\vp_1,\cdots,\vp_n,p_1,\cdots,p_n)$ form a Darboux chart.
		These are called \textbf{action coordinates}.
	\end{enumerate}
\end{thm}

\subsubsection*{Delaunay Coordinates}
First, we introduce Delaunay coordinates, which is widely used for the Kepler problem and the three-body problem.
A derivation is given in the appendix, and we refer \cite{Poincare_87} and \cite{Chenciner_89} for the detailed description of this coordinate system.
\textbf{Delaunay coordinates} are action-angle coordinates with following three Poisson-commuting integrals of motion,
	$$
	(p_l,p_g,p_\theta)=\left(\frac{1}{\sqrt{-2E}},|L|,L_3\right),
	$$
and we denote $(l,g,\theta)$ for the conjugate coordinates.
\begin{note}\rm
	A lot of literatures are using $(L,G,\Omega)$ instead of $(p_l,p_g,p_\theta)$.
\end{note}
Under this coordinate system, our Hamiltonians can be written as
	$$
	E=-\frac{1}{2p_l^2},\quad H=-\frac{1}{2p_l^2} + p_\theta.
	$$
Away from the coordinate singularity, the Hamiltonian flow of $E$ and $H$ are linear,
$$
\begin{aligned}
Fl^{X_E}_t(l_0,g_0,\theta_0)&= \left(l_0 + \frac{1}{p_l^3}t,g_0,\theta_0\right),\\
Fl^{X_H}_t(l_0,g_0,\theta_0)&= \left(l_0 + \frac{1}{p_l^3}t,g_0,\theta_0+t\right),
\end{aligned}
$$
while $(p_l,p_g,p_\theta)$ are constant.

To apply the Arnold-Liouville theorem, we need to show the differentials of $p_l$, $p_g$ and $p_\theta$ are linealy independent.
Consider spherical coordinates on $T^*\R^3$, which are
	$$
	\begin{pmatrix}
		q_1\\
		q_2\\
		q_3
	\end{pmatrix}
	=
	\begin{pmatrix}
		r \sin\psi \cos \vp\\
		r \sin \psi \sin \vp\\
		r \cos \psi
	\end{pmatrix},
	$$
	
	$$
	\begin{pmatrix}
		p_r\\
		p_\psi\\
		p_\vp
	\end{pmatrix}
	=
	\begin{pmatrix}
		\sin\psi \cos\vp &	\sin\psi\sin\vp&	\cos\psi\\
		r\cos\psi\cos\vp&	r\cos\psi\sin\vp&	-r\sin\psi\\
		-r\sin\psi\sin\vp&	r\sin\psi\cos\vp&	0
	\end{pmatrix}
	\begin{pmatrix}
		p_1\\
		p_2\\
		p_3
	\end{pmatrix}.
	$$
Here's a relation between spherical coordinates and Delaunay coordinates.
$$
\begin{pmatrix}
	p_l\\
	p_g\\
	p_\theta
\end{pmatrix}
=
\begin{pmatrix}
	\left(\frac{2}{r}-p_r^2-\frac{p_\psi^2}{r^2}-\frac{p_\vp^2}{r^2\sin^2\psi}\right)^{-1/2}\\
	\left(p_\psi^2 + \frac{p_\vp^2}{\sin^2 \psi}\right)^{1/2}\\
	p_\vp
\end{pmatrix}.
$$
We know that the coordinate change from Euclidean coordinates to spherical coordinates is a well-defined symplectomorphism except for the origin, so it's enough to show the linear independency of the derivatives of $(p_l,p_g,p_\theta)$ in terms of spherical coordinates.
From the direct computation, we have the explicit formula
	$$
\left(\frac{\pp (p_l,p_g,p_\theta)}{\pp(r,\psi,\vp,p_r,p_\psi,p_\vp)}\right)^t
=
\begin{pmatrix}
	-\frac{1}{(-2E)^{3/2}}\left(\frac{p_\psi^2}{r^3} - \frac{1}{r^2} + \frac{p_\vp^2}{r^3\sin^2\psi}\right)
	&
	0
	&
	0
	\\
	-\frac{1}{(-2E)^{3/2}}\frac{p_\vp^2\cos\psi}{r^2\sin^3\psi}
	&
	-\frac{1}{|L|}\frac{p_\vp^2\cos\psi}{\sin^3\psi}
	&
	0
	\\0
	&
	0
	&
	0
	\\\frac{p_r}{(-2E)^{3/2}}
	&
	0
	&
	0
	\\\frac{1}{(-2E)^{3/2}}\frac{p_\psi}{r^2}
	&
	\frac{p_\psi}{|L|}
	&
	0
	\\\frac{1}{(-2E)^{3/2}}\frac{p_\vp}{r^2\sin^2\psi}
	&
	\frac{p_\vp}{|L|\sin^2\psi}
	&
	1
\end{pmatrix}.
$$
We note some properties.
	\begin{itemize}
		\item The collision orbits are excluded, since they are not contained in the domain of spherical coordinates.
		\item $p_r=0$ if $r$ is constant, which means that the orbit is circular.
		\item $p_\psi=0$ if $\psi$ is constant, which means that the orbit is planar.
		\item $p_\vp=0$ if $\vp$ is constant, which means that the orbit is vertical.
		In fact, vertical orbits are not contained in $\Sigma_{k,l}$-family generically, so it's not of our interest.
		Hence we can assume that $p_\vp\neq0$.
	\end{itemize}
We break the situation into three cases.
	\begin{enumerate}
		\item If the orbit is not circular and not planar, we have $p_r\neq0$ and $p_\psi\neq0$.
		It follows that 4th, 5th and 6th row are linearly independent.
		\item If the orbit is circular and not planar, we still have $p_\psi\neq0$.
		We claim that (1,1)-component is nonzero.
		Indeed, we have
		$$
		-\frac{1}{r}+\frac{p_\psi^2}{r^2}+\frac{p_\vp^2}{r^2\sin^2\psi}=-\frac{1}{r} + \left(\frac{2}{r} + p_r^2 - 2E\right) = \frac{1}{r} - 2E.
		$$
		Meanwhile, we have $r=1/\sqrt{-2E}$ for circular orbits, so if $E<-1/2$, this is nonzero.	
		Therefore, 1st, 5th and 6th row are linearly independent.
		\item If the orbit is planar, we have $p_\psi=0$ and $\cos \psi=0$.
		In this case, 2nd and 3rd columns are linearly dependent.
	\end{enumerate}

In short, the Delaunay coordinate system is valid for non-planar orbits.
To see the Morse-Bott property, we need to compute the linearized flow.
Let $E=E_{k,m}=-\frac{1}{2}\left(\frac{k}{m}\right)^{2/3}$, so that the orbits of the rotating Kepler problem with Kepler energy $E$ is periodic with period $2\pi k$.
The linearized Hamiltonian flow of $H$ for the $\Sigma_{k,m}$-orbits in Delaunay coordinates can be computed by following matrix
$$
\mathbf{L}=\begin{pmatrix}
	0&0&0&-3/p_l^4&0&0\\
	0&0&0&0&0&0\\
	0&0&0&0&0&0\\
	0&0&0&0&0&0\\
	0&0&0&0&0&0\\
	0&0&0&0&0&0
\end{pmatrix}.
$$
Let's denote the return time by $\tau=2\pi m = 2\pi k p_l^3$.
Let $\Phi$ be a linearized map, i.e. $\Phi(t) = \mathbf{L}\Phi(0)$.
Since $\mathbf{L}$ is nilpotent, $\Phi$ can be easily computed as $\Phi(t) = \exp \mathbf{L}(t)$.
That is,
$$
\Phi (t) = \begin{pmatrix}
	1&0&0&-3t/p_l^4&0&0\\
	0&1&0&0&0&0\\
	0&0&1&0&0&0\\
	0&0&0&1&0&0\\
	0&0&0&0&1&0\\
	0&0&0&0&0&1
\end{pmatrix}.
$$
The linearized return map is $\Psi=\Phi(\tau)$, which is
$$
\Psi = \begin{pmatrix}
	1&0&0& -6\pi k / p_l & 0&0\\
	0&1&0&0&0&0\\
	0&0&1&0&0&0\\
	0&0&0&1&0&0\\
	0&0&0&0&1&0\\
	0&0&0&0&0&1
\end{pmatrix}.
$$
Our manifold $\Sigma_{k,m}$ is characterized by $E$ and $L_3$, which can be written $p_l=p_{l,0}$ and $p_\theta=p_{\theta,0}$ in Delaunay coordinates.
Let
$$
\nu = -p_l^3 \frac{\pp}{\pp p_l} + \frac{\pp}{\pp p_\theta}
$$
be a normal vector of $\Sigma$.
Note that $\langle \nu, \nabla H \rangle =0$.
We have that
$$
\Psi \nu = 6\pi k p_l^2 \frac{\pp}{\pp l} + \nu,
$$
which means that $\Psi|_{\nu \Sigma}\neq\Id$.
This proves that $\Sigma$ is Morse-Bott for non-planar orbits.

\subsubsection*{LRL Coordinates}

To cover the planar orbits, we introduce another combination of the integrals of motions.
Since $\{A_3,L_3\}=0$, we can take $A_3$ instead of $|L|$.
In spherical coordinates, we can write
	$$
	A_3 = \frac{\cos\psi}{r}\left(p_\psi^2 + \frac{p_\vp^2}{\sin^2\psi}\right) + \sin\psi p_r p_\psi - \cos \psi.
	$$
We take $(p_l,p_\eta,p_\theta)$ for new coordinates, where $p_\eta=A_3$.
We call $(l,\eta,\theta,p_l,p_\eta,p_\theta)$ \textbf{Laplace-Runge-Lenz coordinates}.
Then we have
$$
\frac{\pp p_\eta}{\pp(r,\psi,\theta,p_r,p_\psi,p_\vp)}^t
=\begin{pmatrix}
	-\frac{\cos \psi (p_\psi^2 + p_\vp^2 /\sin^2 \psi)}{r^2}\\
	p_r p_\psi \cos\psi + \sin\psi -\frac{2p_\vp^2 \cos\psi}{r\sin^3 \psi} - \frac{\sin\psi (p_\psi^2 + p_\vp^2/\sin^2\psi)}{r}\\
	0\\
	p_\psi \sin\psi\\
	p_r\sin\psi + \frac{2p_\psi \cos\psi}{r}\\
	\frac{2p_\vp \cos\psi}{r \sin^2\psi}
\end{pmatrix}.
$$
Consider the non-circular planar orbits determined by 
$$
p_r\neq0,\quad p_\psi =0,\quad \cos\psi=0,\quad \sin\psi=1.
$$
Then we have
$$
\left(\frac{\pp (p_l,p_\eta,p_\theta)}{\pp(r,\psi,\vp,p_r,p_\psi,p_\vp)}\right)^t
=
\begin{pmatrix}
	-\frac{1}{(-2E)^{3/2}}\left( - \frac{1}{r^2} + \frac{p_\vp^2}{r^3}\right)
	&
	0
	&
	0
	\\
	0
	&
	\sin\psi\left( 1 - \frac{p_\vp}{r}\right)
	&
	0
	\\0
	&
	0
	&
	0
	\\\frac{p_r}{(-2E)^{3/2}}
	&
	0
	&
	0
	\\
	0&
	p_r &
	0
	\\\frac{1}{(-2E)^{3/2}}\frac{p_\vp}{r^2}
	&
	\frac{p_\vp}{|L|}
	&
	1
\end{pmatrix}.
$$
We have that the 4th, 5th and 6th row are linearly independent, and the flow is independent of $p_\eta$.
Therefore, we can apply the same argument as the previous section to show the Morse-Bott property for the non-cicrular planar orbits.
In short, we can establish the following proposition.
	\begin{prop}\label{Proposition - Morse-Bott property}
		Assume that $c\neq E_{k,l}$ and $c\neq E_{k,l}\pm1/\sqrt{-2E_{k,l}}$ for any $k,l\in\N$.
		Then the $S^3$-family $\Sigma_{k,l}$, which consists of the periodic orbits of the rotating Kepler problem with Kepler energy $E_{k,l}$ and Jacobi energy $c$, has the Morse-Bott property.
	\end{prop}

\subsection{Indices of $\Sigma_{k,l}$-Families}

We've seen that $\Sigma_{k,l}$-families are generically Morse-Bott in \Cref{Proposition - Morse-Bott property}.
We recall the local Floer homology and Morse-Bott spectral sequence which can be utilized in our case.
The \textbf{local Floer homology}, defined in \cite{Cieliebak_Floer_Hofer_Wysocki_96} and also well-explained in \cite{Kwon_van_Koert_16}, computes the Floer homology of a neighborhood of a Morse-Bott submanifold.
\begin{prop}\label{Proposition - local Floer}
	Let $(W,d\ld)$ be a Liouville domain and $\Sigma$ be a Morse-Bott component of a Morse-Bott type Hamiltonian $H:W\to\R$ and $H(\Sigma)=c$.
	If
		\begin{enumerate}
			\item (Symplectic triviality) there exists a symplectic trivialization for each critical component which can be extended to $W$,
			\item $c_1(W)=0$,
			\item $c$ is regular value and the Liouville vector field is transversal to $H^{-1}(c)$,
			\item the Hessian of $H$ to the direction of Liouville vector field is positive definite,
		\end{enumerate}
	then the local Floer homology of $\Sigma$ with $\Z_2$-coefficient is isomorphic to the singular homology of $\Sigma$ with degree shift, i.e.,
		$$
		HF^{loc}_{*+sh(\Sigma)}(\Sigma,H,J;\Z_2)\simeq H_*(\Sigma;\Z_2),
		$$
	where $sh(\Sigma)=\mu_{RS}(\Sigma)-\frac{1}{2}\dim \Sigma/S^1$.
\end{prop}
\begin{proof}
	See Proposition 8.4. of \cite{Kwon_van_Koert_16}.
	We used $\Z_2$-coefficient to avoid the discussion about the orientation.
\end{proof}

\begin{note}\rm
	When we apply \Cref{Proposition - local Floer} to our situation, the Hamiltonian appearing in that theorem is not the Kepler Hamiltonian itself, but rather a modified Hamiltonian that has either the retrograde or the direct orbit and $\Sigma_{k,l}$-family as a critical component and is adjusted so as to behave well on the cylindrical end.
	Indeed, the positive definiteness assumed in the theorem refers to the positive definiteness of this modified Hamiltonian, not that of the Kepler Hamiltonian (which is in fact not positive definite), and this modified Hamiltonian can always be chosen so that the required condition is satisfied.
\end{note}

Using the filtration by periods, we can construct a spectral sequence whose first page consists of the local Floer homology of the critical components which computes the Floer homology of given Liouville domain $W$.
Precisely, the first page is given by following.
	$$
	E_{pq}^1(HF(W;\Z_2)) = \left\{\begin{array}{cc}
		\bigoplus_{\Sigma\in C(p)}H_{p+q-sh(\Sigma)}(\Sigma;\Z_2)&\text{if }p>0,\\
		H_{q+n}(W,\pp W)&\text{if }p=0,\\
		0&\text{otherwise}.
	\end{array}
		\right.
	$$
We call this \textbf{Morse-Bott spectral sequence}.
See Chapter 8 of \cite{Kwon_van_Koert_16} for the detailed definition.

Using the local Floer homology, we can compute the degree shift and Robbin-Salamon indices of the $\Sigma_{k,l}$-families.
To understand the situation, we consider the following simple example illustrated in \Cref{Figure : MBSS simple}.
We write
	$$
	c_{k,l}^\pm = E_{k,l} \pm \frac{1}{\sqrt{-2E_{k,l}}}
	$$
for the Jacobi energy of the retrograde and direct orbits, at which the bifurcation occurs.
At $c_- = c^-_{8,1}-\eps$, the seventh cover of the direct orbit $\g_-^7$ has index 30, and $\Sigma_{8,1}$-family does not exist.
If the Jacobi energy passes through $c^-_{8,1}$ and become $c_+=c^-_{8,1}+\eps$, the index of $\g_-^7$ becomes 34 and $\Sigma_{8,1}$-family is born.
Both cases should give the same result due to the invariance of the local Floer homology, and from Morse-Bott spectral sequence, one can deduce that the degree shift of $\Sigma_{8,1}$-family must be 30.
The following theorem gives the rigorous explanation of the general case.

\begin{figure}[ht]
	\centering
	\includegraphics[width=7cm]{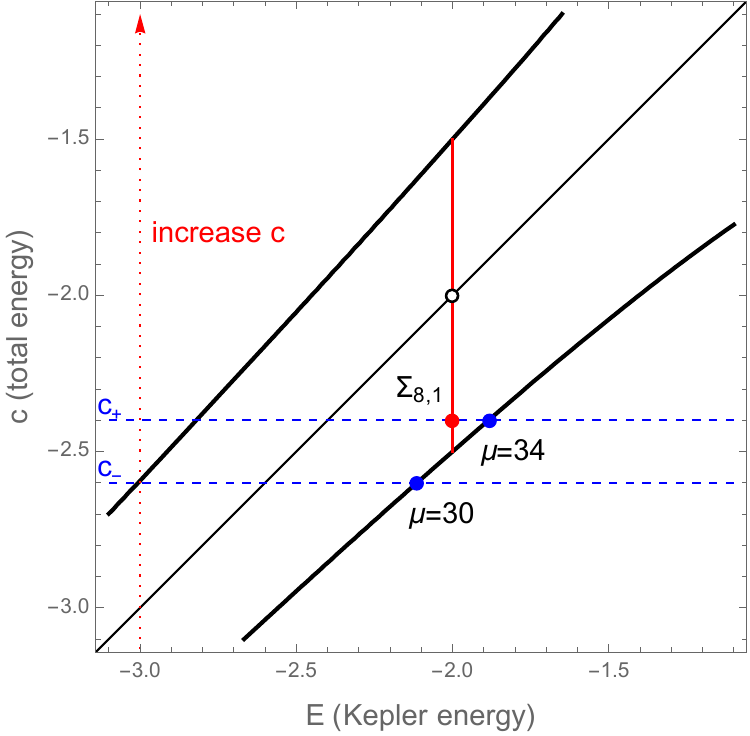}
	\caption{Bifurcation at $c^-_{8,1}$.}
	\label{Figure : MBSS simple}
\end{figure}

\begin{thm}\label{Theorem - Index of Degenerate Orbits}
	The Robbin-Salamon index of $\Sigma_{k,l}$-family is $4k-1/2$.
\end{thm}
\begin{proof}
	As mentioned in the last part of \Cref{Subsection - Moduli}, $\Sigma_{k,l}$-family becomes singular at $L_3=0$.
	So we divide the proof into two cases.
	\subsubsection*{Case 1: The $\Sigma_{k,l}$-family with $L_3 < 0$}
	The first case is $\Sigma_{k,l}$-family with $L_3<0$.
	In this case we consider the direct orbit. As in the previous example, we perturb the Jacobi energy so that it passes through $c_{k,l}^-$.  
	By \Cref{Theorem - Index of planar circular orbits}, we have
	$\mu_{CZ}(\g_-^{k-l})$ is $4k - 2$ at $c_{k,l}^- - \eps$,
and $4k + 2$ at $c_{k,l}^- + \eps$.
	
	We now consider two spectral sequences computing the local Floer homology of a neighborhood of $\g_-^{k-l}$.  
	Since we are on the cylindrical end, multiple covers of an orbit are separated, so there is no need to consider additional coverings.
	At the energy level $c_{k,l}^- - \eps$, the only generator is $\g_-^{k-l}$ with degree shift $4k - 2$, as illustrated in \Cref{Figure - MBSS1}.  
	We omit the $p=0$ column for simplicity; in particular, the degree of
	$
	H_{*}(\g_-^{k-l}) \simeq H_*(S^1)
$	starts at
	$sh(\g_-^{k-l}) - p
	= \mu_{CZ}(\g_-^{k-l}) - 1
	= 4k - 3.
	$
	
	At the energy $c_{k,l}^- + \varepsilon$, the degree shift of $\g_-^{k-l}$ changes accordingly.  
	Here, the $\Sigma_{k,l}$-family also contributes with degree shift denoted by $N$.  
	Since the period of $\g_-^{k-l}$ is longer than that of $\Sigma_{k,l}$, we place $\Sigma_{k,l}$ at $p=1$ and $\g_-^{k-l}$ at $p=2$.
	
	At $p=2$, we have $H_*(S^1)$ starting at $sh(\g_-^{k-l}) - p= 4k + 2 - 2 = 4k$.
	At $p=1$, we have
	$H_*(S^3 \times S^1)$ starting at $sh(\Sigma_{k,l}) - 1 = N - 1$.
	There are six generators in total.  
	Therefore, compared with the configuration at $c_{k,l}^- - \eps$, the spectral sequence must contain a differential of degree $(1,0)$ that cancels two pairs.  
	This yields two possibilities: $N-1=4k-3$ or $4k$.
	
	The second option is excluded, since it leaves generators in total degrees $4k+4$ and $4k+5$, which do not match the configuration for $c_{k,l}^- - \eps$.  
	Thus we conclude $N - 1 = 4k - 3$, and hence
	$$	\mu_{RS}(\Sigma_{k,l})
	=sh(\Sigma_{k,l})+\frac{1}{2} \dim (\Sigma_{k,l}/S^1)
	= (4k - 2) + \frac{3}{2}
	= 4k - \frac{1}{2}.
	$$
	The situation is illustrated in \Cref{Figure - MBSS2}.
	
	\subsubsection*{Case 2: The $\Sigma_{k,l}$-family with $L_3 > 0$}
	
	In this case we consider the retrograde orbit.  
	Let the Jacobi energy pass through $c_{k,l}^+$.  
	Again by \Cref{Theorem - Index of planar circular orbits},
	$\mu_{CZ}(\g_+^{k+l}) = 4k + 2$ at $c_{k,l}^+ - \eps$ and $4k - 2$ at $c_{k,l}^+ + \eps$.
	Moreover, the $\Sigma_{k,l}$-family exists for $c_{k,l}^+ - \eps$ and disappears at $c_{k,l}^+ + \eps$.  
	Applying the same spectral–sequence argument as above, the degree shift must again be $4k - 2$, yielding the desired result.  
	The spectral sequences for this case are shown in \Cref{Figure - MBSS3} and \Cref{Figure - MBSS4}.
	
\end{proof}

\begin{figure}[h]
	\centering
	
	\begin{minipage}{0.23\textwidth}
		\centering
		\begin{tikzpicture}[scale=0.6]
			\draw (0.5,0.5) -- (0.5,5.5);
			\draw (0.5,0.5) -- (1.5,0.5);
			
			\node[left] at (0.4,1) {$4k-3$};
			\node[left] at (0.4,2) {$4k-2$};
			\node[left] at (0.4,3) {$4k-1$};
			\node[left] at (0.4,4) {$4k$};
			\node[left] at (0.4,5) {$4k+1$};
			
			\node[below] at (1,0.5) {\shortstack{1 \\ $\gamma_-^{k-l}$}};
			
			\filldraw[black] (1,1) circle (2pt);
			\filldraw[black] (1,2) circle (2pt);
			
		\end{tikzpicture}
		\subcaption{At $c_{k,l}^--\varepsilon$}
		\label{Figure - MBSS1}
	\end{minipage}
	\hfill
	\begin{minipage}{0.23\textwidth}
		\centering
				\begin{tikzpicture}[scale=0.6]
			\draw (0.5,0.5) -- (0.5,5.5);
			\draw (0.5,0.5) -- (2.5,0.5);
			
			\node[left] at (0.4,1) {$4k-3$};
\node[left] at (0.4,2) {$4k-2$};
\node[left] at (0.4,3) {$4k-1$};
\node[left] at (0.4,4) {$4k$};
\node[left] at (0.4,5) {$4k+1$};
			
			\node[below] at (1,0.5) {\shortstack{1 \\ $\Sigma_{k,l}$}};
			\node[below] at (2,0.5) {\shortstack{2 \\ $\gamma_-^{k-l}$}};
			
			\filldraw[red] (1,1) circle (2pt);
			\filldraw[red] (1,2) circle (2pt);
						\filldraw[red] (1,4) circle (2pt);
			\filldraw[red] (1,5) circle (2pt);
						\filldraw[black] (2,4) circle (2pt);
						\filldraw[black] (2,5) circle (2pt);
			
						\draw[black] (1,4) -- (2,4);
						\draw[black] (1,5) -- (2,5);
			
		\end{tikzpicture}
		\subcaption{At $c_{k,l}^-+\varepsilon$}
				\label{Figure - MBSS2}
	\end{minipage}
	\hfill
	\begin{minipage}{0.23\textwidth}
		\centering

\begin{tikzpicture}[scale=0.6]
	\draw (0.5,0.5) -- (0.5,5.5);
	\draw (0.5,0.5) -- (2.5,0.5);
	
			\node[left] at (0.4,1) {$4k-3$};
\node[left] at (0.4,2) {$4k-2$};
\node[left] at (0.4,3) {$4k-1$};
\node[left] at (0.4,4) {$4k$};
\node[left] at (0.4,5) {$4k+1$};
	
	\node[below] at (1,0.5) {\shortstack{1 \\ $\Sigma_{k,l}$}};
	\node[below] at (2,0.5) {\shortstack{2 \\ $\gamma_-^{k-l}$}};
	
	\filldraw[red] (1,1) circle (2pt);
	\filldraw[red] (1,2) circle (2pt);
	\filldraw[red] (1,4) circle (2pt);
	\filldraw[red] (1,5) circle (2pt);
	\filldraw[black] (2,4) circle (2pt);
	\filldraw[black] (2,5) circle (2pt);
	
	\draw[black] (1,4) -- (2,4);
	\draw[black] (1,5) -- (2,5);
	
\end{tikzpicture}
		\subcaption{At $c_{k,l}^+ - \varepsilon$}
				\label{Figure - MBSS3}
	\end{minipage}
	\hfill
	\begin{minipage}{0.23\textwidth}
		\centering
		\begin{tikzpicture}[scale=0.6]
	\draw (0.5,0.5) -- (0.5,5.5);
	\draw (0.5,0.5) -- (1.5,0.5);
	
			\node[left] at (0.4,1) {$4k-3$};
\node[left] at (0.4,2) {$4k-2$};
\node[left] at (0.4,3) {$4k-1$};
\node[left] at (0.4,4) {$4k$};
\node[left] at (0.4,5) {$4k+1$};
	
	\node[below] at (1,0.5) {\shortstack{1 \\ $\gamma_-^{k-l}$}};
	
	\filldraw[black] (1,1) circle (2pt);
	\filldraw[black] (1,2) circle (2pt);
	
\end{tikzpicture}
		\subcaption{At $c_{k,l}^+ + \varepsilon$}
				\label{Figure - MBSS4}
	\end{minipage}
	
	\caption{Morse-Bott spectral sequences near $c_{k,l}^\pm$.}
	\label{Figure - MBSS}
\end{figure}
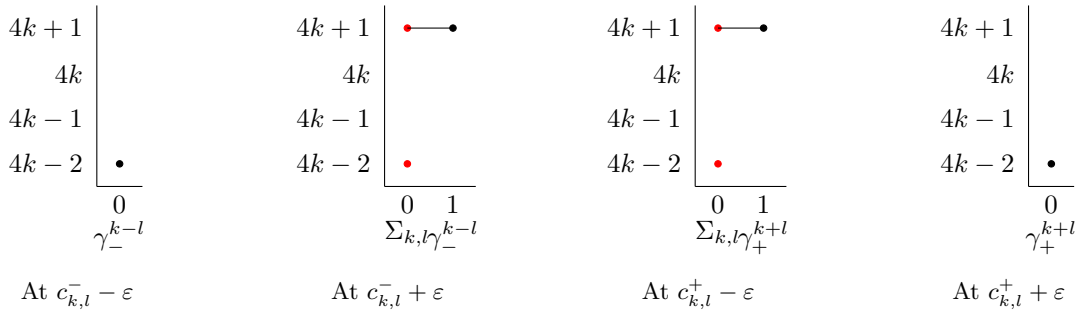



\section{Appendix - Delaunay Coordinates}\label{Appendix - Delaunay coordinates}

	\subsection{Hamilton-Jacobi Equation}
	
	We refer \cite{Arnold_89} for a detailed description of Hamilton-Jacobi equation.
	Let $(p,q)$ be coordinates in $T^*\R^n$.
	Let $g:T^*\R^n\to T^*\R^n$ be an exact symplectomorphism, which is written by $g(p,q)=(P,Q)$.
	There exists a function $S=S(p,q)$, which is called a \textbf{generating function}, such that
	$$
	pdq - PdQ = dS.
	$$
	Assume that $\det \frac{\pp(Q,q)}{\pp(p,q)}\neq0$, so that $S(p,q)=S_1(Q,q)$.
	Then we have
	$$
	\frac{\pp S_1}{\pp q} = p,\,\,\frac{\pp S_1}{\pp Q} = -P.
	$$
	We're looking for a nice coordinate system to describe the given Hamiltonian.
	If we can write $H(p,q)=K(Q)$ for some function $K$, the Hamiltonian equation becomes
	$$
	\dot{Q}=0,\,\,\dot{P}=\frac{\pp K}{\pp Q},
	$$
	which makes the flow linear,
	$$
	(Q(t),P(t)) = \left(Q(0),P(0)+\left.\frac{\pp K}{\pp Q}\right|_{Q(0)}t\right).
	$$
	That is, we need to find a generating function $S$ such that
	$$
	H\left(\frac{\pp S(Q,q)}{\pp q},q,t\right)=K(Q).
	$$
	If such relation holds, each component of $Q$ is the first integral of $H$.
	This equation is called a \textbf{Hamilton-Jacobi equation}.
	
	\subsection{Derivation of Delaunay Coordinates}
			
	Under spherical coordinates, the Kepler Hamiltonian is
	$$
	E = \frac{1}{2}\left(p_r^2 + \frac{p_\psi^2}{r^2}+\frac{p_\vp^2}{r^2\sin^2\psi}\right)-\frac{1}{r}.
	$$	
	To find action-angle coordinates, we introduce new variables $(p_l,p_g,p_\theta)$ in the place of $Q$ in the Hamilton-Jacobi equation.
	The equation becomes
	$$
	\frac{1}{2}\left(\left(\frac{\pp S}{\pp r}\right)^2 + \frac{1}{r^2}\left(\frac{\pp S}{\pp \psi}\right)^2 + \frac{1}{r^2\sin^2\psi}\left(\frac{\pp S}{\pp \vp}\right)^2\right) - \frac{1}{r} = E.
	$$
	The generating function $S$ should be written in terms of $(r,\psi,\vp,p_l,p_g,p_\theta)$.
	Assume that $S$ can be split into three functions,
	$$
	S(r,\psi,\vp,p_l,p_g,p_\theta) = S_1(r,p_l,p_g) + S_2(\psi,p_g,p_\theta) + S_3(\vp,p_\theta).
	$$
	We first define $p_l$ by
	$$
	E = -\frac{1}{2p_l^2},
	$$
	so that, if $(p_l,p_g,p_\theta)$ is a solution of Hamilton-Jacobi equation, then they are first integrals.
	Then we find other solutions as follows.
	$$
	\begin{aligned}
		\left(\frac{\pp S_3}{\pp \vp}\right)^2 & = p_\theta^2,\\
		\left(\frac{\pp S_2}{\pp \psi}\right)^2& + \frac{p_\theta^2}{\sin^2\psi} = p_g^2,\\
		\left(\frac{\pp S_1}{\pp r}\right)^2& + \frac{p_g^2}{r^2} - \frac{2}{r} = 2E = -\frac{1}{p_l^2}.
	\end{aligned}
	$$
	The generating function $S=S(r,\psi,\vp,p_l,p_g,p_\theta)$ is
	$$
	S = \int \sqrt{-\frac{1}{p_l^2}+\frac{2}{r}-\frac{p_g^2}{r^2}}dr
	+\int \sqrt{p_g^2 - \frac{p_\theta^2}{\sin^2\psi}}d\psi
	+p_\theta \vp.
	$$
	The conjugate coordinates $(l,g,\theta)$ is given by
	$$
	(l,g,\theta) = \left(\frac{\pp S}{\pp p_l},\frac{\pp S}{\pp p_g},\frac{\pp S}{\pp p_\theta}\right).
	$$
	This coordinate system $(l,g,\theta,p_l,p_g,p_\theta)$ is called \textbf{Delaunay coordinate system}.
	It follows that the Hamiltonian flow is linear
	$$
	(l,g,\theta) = \left(l_0 + \frac{1}{p_l^3}t,g_0,\theta_0\right)
	$$
	while $(p_l,p_g,p_\theta)$ are constant.
	
	\subsection{Geometric Interpretation}
	
	Another relation we can get from the generating function is
	$$
	\begin{aligned}
		(p_r,p_\psi,p_\vp)&=\left(\frac{\pp S}{\pp r},\frac{\pp S}{\pp \psi},\frac{\pp S}{\pp \vp}\right)\\
		&=\left(\sqrt{-\frac{1}{p_l^2}+\frac{2}{r}-\frac{p_g^2}{r^2}},\sqrt{p_g^2-\frac{p_\theta^2}{\sin^2\psi}},p_\theta\right).
	\end{aligned}
	$$
	Let $a$ be a half of the length of the major axis and $\eps$ be the eccentricity of a Kepler orbit.
	Then $[a(1-\eps), a(1+\eps)]$ is the interval which $r$-coordinate is projected to.
	At endpoints, we have $R=0$.
	Thus $a(1\pm \eps)$ are the root of the equation
	$$
	r^2 - 2p_l^2 r + p_g^2 p_l^2 =0.
	$$
	It follows that
	$$
	\begin{aligned}
		p_l^2 & = a,\\
		p_g^2 & = \frac{a^2(1-\eps^2)}{p_l^2} = a(1-\eps^2).
	\end{aligned}
	$$
	We have a formula for the eccentricity of Kepler orbit,
	$$
	\eps^2 = |A|^2 = 2EL^2 + 1 = -\frac{L^2}{p_l^2} + 1 = -\frac{L^2}{a} +1,
	$$
	where $L$ is the angular momentum and $A$ is the Laplace-Runge-Lenz vector.
	Also, since $p_\theta=p_\vp$ and we can easily see from the description of spherical coordinates that $p_\vp$ is $z$-component of the angular momentum.
	It follows that
	$$
	p_l=\sqrt{a}=\frac{1}{\sqrt{-2E}},\quad		p_g= |L|,\quad p_\theta = L_3.
	$$
	The geometric interpretation of the conjugate coordinates are well-known in many literatures in the physics and astronomy, and the derivations can be found in \cite{Chenciner_89}.
	\begin{itemize}
		\item $l$ is the \textbf{mean anomaly}, which measures the angular velocity of the circular orbit which has the same Kepler energy with the given orbit.
		\item $g$ is the \textbf{argument of perigee}, which measures the angle between $q_1q_2$-plane and the periapsis of the ellipse.
		\item $\theta$ is the azimuthal angle in spherical coordinates.
	\end{itemize}

\bibliographystyle{amsalpha}
\bibliography{SRKP.bib}
\end{document}